\documentclass[11pt]{amsart}

\usepackage{fullpage}

\usepackage{amsmath, amssymb, mathrsfs, bbm}

\usepackage{tikz}
 \usetikzlibrary{cd}
 \tikzset{
  symbol/.style={
    draw=none,
    every to/.append style={
      edge node={node [sloped, allow upside down, auto=false]{$#1$}}}
      }
      }

\usepackage{amsthm}
	\theoremstyle{definition} 
	\newtheorem{definition}{Definition}[section]
	
	\theoremstyle{plain} 
	\newtheorem{theorem}[definition]{Theorem}
	\newtheorem*{theorem*}{Theorem} 
	\newtheorem{lemma}[definition]{Lemma}
	\newtheorem{proposition}[definition]{Proposition}

	\theoremstyle{remark} 
	\newtheorem{remark}[definition]{Remark}
                    
\usepackage{enumerate,enumitem}

\usepackage{hyperref}  

\renewcommand{\AA}{\mathbb A}

\newcommand{\CC}{\mathbb C}

\newcommand{\FF}{{\mathbb F}}
\newcommand{\GG}{{\mathbb G}}

\newcommand{\PP}{{\mathbb P}}
\newcommand{\QQ}{{\mathbb Q}}
\newcommand{\RR}{{\mathbb R}}

\newcommand{\TT}{{\mathbb T}}

\newcommand{\ZZ}{{\mathbb Z}}

\newcommand{\cC}{{\mathcal C}}
\newcommand{\cD}{{\mathcal D}}
\newcommand{\cE}{{\mathcal E}}
\newcommand{\cF}{{\mathcal F}}

\newcommand{\cH}{{\mathcal H}}

\newcommand{\cL}{{\mathcal L}}
\newcommand{\cM}{{\mathcal M}}

\newcommand{\cO}{{\mathcal O}}
\newcommand{\cP}{{\mathcal P}}

\newcommand{\cS}{{\mathcal S}}
\newcommand{\cT}{{\mathcal T}}
\newcommand{\cU}{{\mathcal U}}
\newcommand{\cV}{{\mathcal V}}
\newcommand{\cW}{{\mathcal W}}
\newcommand{\cX}{{\mathcal X}}
\newcommand{\cY}{{\mathcal Y}}
\newcommand{\cZ}{{\mathcal Z}}

\newcommand{\frakI}{\mathfrak I}

\newcommand{\frakS}{\mathfrak S}
\newcommand{\frakT}{\mathfrak T}

\newcommand{\frakV}{\mathfrak V}
\newcommand{\frakW}{\mathfrak W}
\newcommand{\frakX}{\mathfrak X}

\newcommand{\frakc}{\mathfrak c}
\newcommand{\frakd}{\mathfrak d}

\newcommand{\frakl}{\mathfrak l}

\newcommand{\frakn}{\mathfrak n}
\newcommand{\frako}{\mathfrak o}
\newcommand{\frakp}{\mathfrak p}
\newcommand{\frakq}{\mathfrak q}

\newcommand{\sW}{\mathscr W}

\newcommand{\Qbar}{\overline{\QQ}}

\newcommand{\Qpbar}{\Qbar_p}

\newcommand{\HT}{\mathrm{HT}}
\newcommand{\dR}{\mathrm{dR}}

\usepackage{colonequals}
\newcommand{\defeq}{\colonequals} 

\newcommand{\isom}{\cong} 

\DeclareMathOperator\Hom{Hom} 
\DeclareMathOperator\End{End} 
\DeclareMathOperator\GL{GL} 

\DeclareMathOperator\Spf{Spf} 
\DeclareMathOperator\Spa{Spa} 

\newcommand{\Gm}{\mathbb{G}_m} 

\newcommand{\tor}{\mathrm{tor}} 

\DeclareMathOperator\Res{Res} 

\newcommand{\id}{\mathrm{id}} 

\DeclareMathOperator\val{val} 
\DeclareMathOperator\Frac{Frac} 
\DeclareMathOperator\Tr{Tr} 
\DeclareMathOperator\Nm{Nm} 
\let\det\relax
\DeclareMathOperator{\det}{det} 

\newcommand{\an}{\mathrm{an}}

\newcommand{\fs}{\mathrm{fs}}

\DeclareMathOperator{\Hdg}{Hdg}
\DeclareMathOperator{\Fitt}{Fitt}
\DeclareMathOperator{\Isom}{Isom}
\newcommand{\RG}{\mathrm{R\Gamma}}
\DeclareMathOperator\rH{H}

\newcommand{\Cl}{\mathscr{C}\!\ell_F^+}

\begin{document}
\title{Eigenvariety for partially classical Hilbert modular forms}
\author{Mladen Dimitrov}
\address{Laboratoire Paul Painlev\'e, Universit\'e de Lille, 59000 Lille, France}
\email{mladen.dimitrov@univ-lille.fr}

\author{Chi-Yun Hsu}
\address{Department of Mathematics and Computer Science, Santa Clara University, Santa Clara, CA 95053, USA}
\email{chsu6@scu.edu}

\date{\today}

\dedicatory{Dedicated to the memory of Jo\"{e}l Bella\"{i}che}   


\begin{abstract}
For each subset of primes in a totally real field  above  a rational prime  $p$, there is the notion of partially classical Hilbert modular forms, where the empty set  recovers the overconvergent forms and the full set of primes above $p$ yields classical forms.
Given such a set, we $p$-adically interpolate the classical modular sheaves to construct families of partially classical Hilbert modular forms with weights varying in appropriate weight spaces and construct the corresponding eigenvariety, generalizing the construction of Andreatta, Iovita, Pilloni, and Stevens. 
\end{abstract}

\maketitle

\addtocontents{toc}{\setcounter{tocdepth}{0}}

\section*{Introduction}
Coleman and Mazur \cite{CM98} constructed the eigencurve, a rigid analytic curve over $\QQ_p$  parametrizing finite-slope Hecke eigensystems  appearing in the space of overconvergent elliptic modular forms of tame level one. It is endowed with a natural map to  a one dimensional weight space parametrizing continuous characters on $\ZZ_p^\times$. 
 
A general method for constructing eigenvarieties, known as the \emph{eigenvariety machine}, was formulated by Buzzard \cite{Buz07}.
Roughly speaking, given an orthonormalizable Banach module $\cM$ over an admissibly covered rigid analytic space $\cW$ and a commutative algebra $\cH$ acting on $\cM$ with a distinguished compact operator $U\in \cH$,  the eigenvariety machine returns an equidimensional eigenvariety $\cE$ with a weight map $\cE\rightarrow \cW$.
The Banach modules, whose importance goes beyond their auxiliary role in the construction of $\cE$ as highlighted by Bella\"{i}che \cite{bellaiche-TJM},  come in  different flavors. 
\begin{enumerate}
 \item {\bf Coherent cohomology of degree zero}.   The theory of canonical subgroups allows one to $p$-adically interpolate classical modular sheaves and then take global sections. This was done by Andreatta, Iovita, Pilloni, and Stevens in the case of Siegel modular forms \cite{AIP_Siegel} and Hilbert modular forms \cite{AIP16,AIS}, by Brasca \cite{Brasca} in the case of PEL Shimura varieties with dense ordinary locus, and by Hernandez \cite{Hernandez} for Picard modular forms. The key technical ingredient is the vanishing of higher cohomologies coming from the affineness of the ordinary locus, which implies that the relevant cohomology is concentrated in the $\rH^0$. 
 For Hilbert modular forms, there is also a second method of $p$-adically interpolating the classical modular sheaves, using the Hilbert modular varieties at infinite level and the Hodge--Tate period map \cite{BHW}. 
 \item {\bf Betti cohomology of distribution sheaves}.  The spreading of Betti cohomology over multiple degrees 
 requires a mechanism beyond Buzzard's eigenvariety machine. This is done for reductive groups whose real points admit discrete series by Urban \cite{Urb} by interpolating the alternating trace on the total cohomology, or by 
 Hansen \cite{Hansen} where the resulting eigenvariety might not be equidimensional. 
 In the case of Hilbert modular forms, an equidimensional eigenvariety using the middle degree cohomology was constructed  by Barrera--Dimitrov--Jorza \cite{BDJ} and by Bergdall--Hansen \cite{BH_Hilb}, 
 both  based on slight variations of Stevens' overconvergent cohomology.  
\end{enumerate}

The eigenvarieties mentioned so far, aside from \cite{BDJ}, focus on the $p$-adic interpolation of classical automorphic forms over a weight space of maximal dimension, where the distinguished operator $U$ is the Hecke operator with respect to the Iwahori subgroup at $p$. 
More generally, one would be interested in the $p$-adic interpolation of classical automorphic forms with respect to other parahoric subgroups, which can be motivated by constructing families of forms having infinite slope at some of the places above $p$. 
This was first investigated by Loeffler \cite{Loeff_oc} in the situation of reductive groups whose real points are compact modulo center. In this situation the associated modular variety is zero-dimensional, so the theory has no geometric input. The case of Hilbert modular forms was worked out in \cite{BDJ}. Further research 
for representations of  $\GL(2n)$ admitting a Shalika model was taken up by Barrera--Dimitrov--Williams, where one considers the parahoric subgroup corresponding to the maximal $(n,n)$-parabolic.  
More generally, Barrera--Williams \cite{BS-W} constructed parabolic eigenvarieties via overconvergent cohomology. 

In this paper, we are interested in the construction of partial Hilbert eigenvarieties using coherent cohomologies. 
We $p$-adically interpolate the classical modular sheaves, following the approach of Andreatta--Iovita--Pilloni.
Because the partial ordinary locus is not affine, the $\rH^0$ cannot be proven to be orthonormalizable.
This prevents one from applying Buzzard's eigenvariety machine to construct the partial eigenvariety.

To address this issue, we propose an alternative path inspired by the Ash--Stevens method.
We consider the total cohomology complex of the modular sheaf. 
There is a natural admissible covering of the partial ordinary locus by affinoids, obtained by pulling back a natural covering of the flag varieties via the Hodge--Tate period map,  which is stable by the relevant Hecke operators. 
The \v{C}ech complex associated to the admissible covering consists of projective modules endowed with a compact operator and is quasi-isomorphic to the total cohomology complex. The Ash--Stevens method, as adapted by Boxer--Pilloni \cite{BP}, allows one to
deduce that the total cohomology admits a slope decomposition and to  construct the corresponding eigenvariety. 

While this works sets the stage for studying the $p$-adic variations of partially classical Hilbert modular forms, including the fascinating  case of partial weight $1$, it also leaves many natural questions unanswered. For example, it 
remains unclear whether every partially classical form appears in the partial eigenvariety, namely whether it deforms in a partial family. We hope to study this and related questions in future work.

\tableofcontents

\subsection*{Notation}
Let  $F$ be a  totally real  number field of degree $g$, with ring of integers $\frako$ and different $\frakd$, and  
fix an odd  prime number $p$ which is unramified in $F$. 
Let $\Sigma$ denote the set of Archimedean embeddings of $F$ which we regard as $p$-adic embeddings via a fixed isomorphism $\iota_p\colon\CC\xrightarrow{\sim}\Qpbar$. 
Letting $\Sigma_p$ denote the set of primes $\frakp$ of $F$ above $p$, there is a partition 
$\Sigma=\coprod_{\frakp\in \Sigma_p}\Sigma_{\frakp}$ where $\Sigma_{\frakp}$ is the subset of $p$-adic embeddings inducing $\frakp$. Clearly $ f_\frakp=|\Sigma_\frakp|$ is the residual degree of $\frakp$.
More generally, for any $I\subseteq \Sigma_p$ we let $\Sigma_I = \bigcup_{\frakp\in I} \Sigma_\frakp$. Finally, for $\underline{n} = (n_\frakp) \in \ZZ_{\geqslant 0}^{\Sigma_p}$, we  
denote by $p^{\underline{n}}$ the ideal $\prod_{\frakp\in \Sigma_p} \frakp^{n_\frakp}$.

 \addtocontents{toc}{\setcounter{tocdepth}{2}}

\section{Partial canonical subgroups for Hilbert modular varieties}

\subsection{Hilbert modular varieties} \label{HMV} \

Let $G = \Res_{F/\QQ}\GL_2$ and $G^\ast\defeq G\times_{\Res_{F/\QQ}\Gm}\Gm$.
We first recall the Shimura variety attached to $G^\ast$ which is of PEL type and has a nice moduli interpretation. 
We then explicitly describe the Shimura variety attached to $G$ as the quotient of the Shimura variety of $G^\ast$ by a finite group. 

Let $\frakn\subseteq \frako$ be a non-zero ideal.
Let $\mu_\frakn$ be the Cartier dual of 
$\frako/\frakn$ seen as a finite flat group scheme over $\ZZ$. Let $\frakc$ be a fractional ideal of $F$, and $\frakc^+\subseteq\frakc$ be the cone of totally positive elements, i.e., the elements in $\frakc$ which are positive under every embedding $\tau\colon F\rightarrow \RR$.
When $\frakn$ is sufficiently small, there exists a smooth quasi-projective Hilbert modular $\QQ_p$-scheme $X_\frakc$ classifying $\underline{A} = (A,i,\lambda,\alpha)$ 
over a $\QQ_p$-scheme $S$, where: 
\begin{itemize}
    \item $A$ is an abelian scheme over $S$ of relative dimension $g$.
    \item $i\colon \frako \hookrightarrow \End_S(A)$ is a ring homomorphism.
    \item $\lambda\colon (\cP_A, \cP_A^+) \rightarrow (\frakc,\frakc^+)$ is  a $\frakc$-polarization, i.e.,  an isomorphism of $\frako$-modules preserving positivity, and inducing an isomorphism $A\otimes_{\frako} \frakc \isom A^\vee$.
    Here $\cP_A$ is the projective rank $1$ $\frako$-module  consisting of symmetric morphisms from $A$ to its dual abelian scheme $A^\vee$, and $\cP_A^+\subseteq \cP_A$ is the cone of polarizations.
    \item $\alpha\colon \mu_\frakn  \hookrightarrow A$ is a closed immersion of group schemes compatible with the $\frako$-actions,  called a $\Gamma_1^1(\frakn)$-level structure. 
\end{itemize} 
The modular variety $X_{\frakd^{-1}}$ is the Shimura variety attached to $G^\ast$.
Henceforth we  assume that $p$ is relatively prime to $\frakn$.
Then there exists an integral model of $X_\frakc$ over $\ZZ_p$, which we continue to denote by $X_\frakc$.
The group  $\frako_{+}^\times$ of totally positive units of $\frako$ acts  on $X_\frakc$ as follows: given $\epsilon\in\frako_{+}^\times$ and an $S$-point $(A/S, i, \lambda, \alpha)$ of $X_\frakc$, we have
\[
 \epsilon\cdot (A/S,i,\lambda,\alpha) = (A/S, i, i(\epsilon)\circ\lambda,\alpha).
\]
This action  factors through the finite quotient $\Delta_\frakn\defeq \frako_{+}^\times/E_\frakn^2$, where $E_\frakn\defeq\{\eta\in\frako^\times\colon\eta\equiv 1\bmod\frakn\}$.
Indeed, if $\epsilon = \eta^2$ with $\eta \in E_\frakn$, then
\[
\epsilon\cdot (A/S,i,\lambda,\alpha)
=(A/S,i,i(\eta^2)\circ\lambda,\eta\alpha)
=\eta^\ast(A/S,i,\lambda,\alpha)
\isom (A/S,i,\lambda,\alpha).
\]

Let $\displaystyle   X\defeq\coprod_{\frakc} X_{\frakc}$,
where $\frakc$ runs over a fixed set of representatives of the narrow class group $\Cl$ of $F$. The quotient of $X$ by $\Delta_\frakn$ is an integral model for the Shimura variety attached to $G$ and its cohomology enjoys a natural action of the  Hecke algebra for $G$. 

Henceforth,  we extend the scalars from $\ZZ_p$ to the ring of integers $\cO$ of sufficiently large finite extension $L$ of $\QQ_p$, in particular containing the images of all $p$-adic embeddings of $F$. 

Denote by  $\frakX$  the completion of $X$ along its special fiber  $\overline{X}$, and by $\cX$ the rigid generic fiber of the formal scheme $\frakX$.
We also use this convention of letter styles for other schemes: when $S$ is a scheme over $\cO$, we denote by $\overline{S}$ its special fiber, by $\frakS$ the associated formal scheme, by $\cS$ the rigid generic fiber of $\frakS$, and by $\cS^\an\supset \cS$ the rigid analytic space associated to the generic fiber $S_L$.
Let $X^\tor$ be a toroidal compactification of $X$, and $X^\ast$ its minimal compactification.
We similarly have $\overline{X}^\tor, \frakX^\tor, \cX^\tor$ and $\overline{X}^\ast, \frakX^\ast, \cX^\ast$.

For $\underline{n}  \in \ZZ_{\geqslant 0}^{\Sigma_p}$, we  let $X_0(\frakc,p^{\underline{n}}) \rightarrow X_{\frakc}$ be the Hilbert modular scheme over $L$ classifying $(\underline{A}, (H_{\frakp})_{\frakp})/S$, where $\underline{A}\in X_\frakc(S)$ and for all $\frakp\in \Sigma_p$, $H_{\frakp}\subseteq A[\frakp^{n_\frakp}]$ is an isotropic $\frako$-invariant finite flat subgroup scheme of rank $p^{f_\frakp  n_\frakp}$.
If $n_\frakp \leqslant 1$ for all $\frakp\in \Sigma_p$, then there exists an integral model of $X_0(\frakc,p^{\underline{n}})$ over $\cO$ (\cite[\S3]{Diamond}), which we continue to denote by the same symbol.
Let $X_1(\frakc,p^{\underline{n}})\rightarrow X_{\frakc}$  be the Hilbert modular scheme over $L$ of level $\Gamma_1^1(\frakn p^{\underline{n}})$. 
We have a natural map $\pi\colon X_1(\frakc,p^{\underline{n}}) \rightarrow X_0(\frakc,p^{\underline{n}})$, which is a $(\frako/p^{\underline{n}})^\times$-torsor.
Given an $L$-valued character $\chi$ of $(\frako/p^{\underline{n}})^\times$, we define 
\begin{equation}\label{eq:twist}
\cO_{X_0(\frakc,p^{\underline{n}})}(\chi) = \left(\pi_\ast \cO_{X_1(\frakc,p^{\underline{n}})}\right)[\chi]
\end{equation}
to be the subsheaf of $\pi_\ast \cO_{X_1(\frakc,p^{\underline{n}})}$ on which $(\frako/p^{\underline{n}})^\times$ acts by $\chi$.
Extending the map $\pi$ to  suitable toroidal compactifications $X_1^\tor(\frakc,p^{\underline{n}})\to X_0^\tor(\frakc,p^{\underline{n}})$  allows one to also extend  $\cO_{X_0(\frakc,p^{\underline{n}})}(\chi)$.

Define 
\[
X_0^\tor(p^{\underline{n}}) \defeq \coprod_{\frakc} X_0^\tor(\frakc,p^{\underline{n}}) \text{ and }X_1^\tor(p^{\underline{n}}) \defeq \coprod_{\frakc} X_1^\tor(\frakc,p^{\underline{n}}),
\]
where $\frakc$ runs over  a set of representatives of $\Cl$.
We have the rigid analytic space $\cX_0(p^{\underline{n}})^\an$, $\cX_1(p^{\underline{n}})^\an$, $\cX_0^\tor(p^{\underline{n}})^\an$ and  $\cX_1^\tor(p^{\underline{n}})^\an$ associated to $X_0(p^{\underline{n}})$, $X_1(p^{\underline{n}})$, 
$X_0^\tor(p^{\underline{n}})$ and  $X_1^\tor(p^{\underline{n}})$. Again, when $\underline{n}\leqslant \underline{1}$, 
we also have $\cX_0(p^{\underline{n}})\subseteq\cX_0(p^{\underline{n}})^\an$ and $\cX_0^\tor(p^{\underline{n}}) = \cX_0^\tor(p^{\underline{n}})^\an$, which are respectively the rigid generic fibers of the formal completions of $X_0(p^{\underline{n}})$ and $X_0^\tor(p^{\underline{n}})$  along their special fibers.

Let $\omega_{\widetilde{A}_\frakc}$ be the sheaf of relative invariant differentials of the semi-abelian scheme $\widetilde{A}_\frakc$ over $X_0^\tor(\frakc,p^{\underline{n}})$.
As $p$ is unramified in $F$, the sheaf $\omega_{\widetilde{A}_\frakc}$ is locally free of rank $1$ as an $\frako\otimes \cO_{X_0^\tor(\frakc,p^{\underline{n}})}$-module.
The $\frako$-module structure on $\omega_{\widetilde{A}_\frakc}$ provides a direct sum decomposition
\[\omega_{\widetilde{A}_\frakc} = \bigoplus_{\tau \in \Sigma} \omega_{\frakc,\tau}, \]
where each $\omega_{\frakc,\tau}$ is locally free of rank $1$ as an $\cO_{X_0^\tor(\frakc,p^{\underline{n}})}$-module.

The natural map from the set $\ZZ^\Sigma\times \ZZ$ of weights of $G$ to the set $\ZZ^\Sigma$ of weights of $G^\ast$ sends $(\underline{\nu},w)$ to $\underline{k}$ where $k_\tau=2\nu_\tau+w$.
Given $\underline{k}\in \ZZ^{\Sigma}$, we define, on $X_0^\tor(\frakc,p^{\underline{n}})$, an invertible sheaf 
\[\omega_\frakc^{\underline{k}} = \bigotimes_{\tau\in \Sigma}  \omega_{\frakc,\tau}^{k_{\tau}}.\]
Given an $L$-valued character $\chi$ of $(\frako/p^{\underline{n}})^\times$, we also define, on $X_0^\tor(\frakc,p^{\underline{n}})$, the invertible sheaf
\[\omega_\frakc^{\underline{k}}(\chi) = \omega_\frakc^{\underline{k}}\otimes_{\cO_{X_0^\tor(\frakc,p^{\underline{n}})}} \cO_{X_0^\tor(\frakc,p^{\underline{n}})}(\chi).\]

Let $\cH^1_\frakc\defeq \cH^1_\dR(\widetilde{A}_\frakc/X_0^\tor(\frakc,p^{\underline{n}}))$ be the sheaf of relative de Rham cohomology.
Again the $\frako$-module structure provides a direct sum decomposition: $\cH^1_\frakc = \bigoplus_{\tau\in\Sigma} \cH^1_{\frakc,\tau}$.
Given $(\underline{\nu}, w)\in \ZZ^\Sigma\times \ZZ$, $\epsilon\in\frako_+^\times$ acts on the sheaf $\bigotimes_{\tau\in\Sigma} \omega_{\frakc,\tau}^{2\nu_\tau+w}\otimes (\wedge^2\cH^1_{\frakc,\tau})^{-\nu_\tau}$ as  
\[
(\epsilon\cdot f)(A/S,i,\lambda,\alpha,\omega) = \nu(\epsilon^{-1})f(A/S,i,i(\epsilon)\circ\lambda,\alpha,\omega),
\]
where $\omega$  is  an $\frako\otimes\cO_{X_0^\tor(\frakc,p^{\underline{n}})}$-basis of $\omega_{\widetilde{A}_\frakc}$ and $f$  is a local section of $\bigotimes_{\tau\in\Sigma} \omega_{\frakc,\tau}^{2\nu_\tau+w}\otimes (\wedge^2\cH^1_{\frakc,\tau})^{-v_\tau}$.  
It can  be checked that the $\frako_+^\times$-action factors through $\Delta_\frakn = \frako_+^\times/E_\frakn^2$, allowing one to define a modular sheaf over $X_0^\tor(\frakc,p^{\underline{n}})/\Delta_\frakn$ by taking  $\Delta_\frakn$-invariant:
\[
\omega_{G,\frakc}^{(\underline{\nu},w)} \defeq \left(\bigotimes_{\tau\in\Sigma} \omega_{\frakc,\tau}^{2\nu_\tau+w}\otimes (\wedge^2\cH^1_{\frakc,\tau})^{-\nu_\tau}\right)^{\Delta_\frakn}.
\]
Given $x\in F_+^\times$ and $\frakc$ a fractional ideal of $F$, we have an isomorphism
\[
L_{\frakc,x}\colon
X_0(\frakc,p^{\underline{n}}) \rightarrow X_0(x\frakc,p^{\underline{n}}) , \quad (A/S, i,\lambda,\alpha) \mapsto (A/S, i, x\lambda,\alpha), \]
inducing
\[
L_{\frakc,x}^\ast\colon
\omega_{G,x\frakc}^{(\underline{\nu},w)} \rightarrow \omega_{G,\frakc}^{(\underline{\nu},w)}, \quad f \mapsto \nu(x)f(A/S,i,x\lambda,\alpha,\omega).
\]
We define the modular sheaf over the Shimura variety for $G$ as
\[
\omega_G^{(\underline{\nu},w)} \defeq \bigoplus_{\frakc} \omega_{G,\frakc}^{(\underline{\nu}, w)}/ (L_{\frakc,x}^\ast-\id)_{x\in F^\times_+},
\]
where $\frakc$ runs over all fractional ideals of $F$.

\subsection{Hodge heights and partial degrees} \label{sec:Hodge}
For each $\tau\in\Sigma$, we define a partial Hodge height function $\Hdg_\tau$ on $\cX^\tor$.
Let $h_\tau\in \rH^0(\overline{X}^\tor,\omega^p_{\sigma^{-1}\circ \tau}\otimes \omega^{-1}_\tau)$ be the $\tau$-th partial Hasse invariant, where $\sigma$ is an automorphism of $L$ over $\QQ_p$ lifting the Frobenius $x\mapsto x^p$ modulo $p$. Let $\val_p$ be the valuation on $\CC_p$ normalized so that $\val_p(p)=1$.
Given a rigid point $x\in\cX^\tor$, we define the $\tau$-th partial Hodge height of $x$ as
\[
\Hdg_\tau(x) = \val_p(\tilde{h}_\tau(x))\in\QQ\cap [0,1],
\] 
where $\tilde{h}_\tau$ is any lift of $h_\tau$ to $\cX^\tor$.
For $\underline{v}\in (\QQ\cap[0,1])^\Sigma$, we define $\cX^\tor(\underline{v})$ to be the inverse image of $\prod_{\tau\in\Sigma}(\QQ\cap[0,v_\tau])$ under the map $(\Hdg_\tau)_\tau\colon \cX^\tor\rightarrow (\QQ\cap [0,1])^\Sigma$.
Consider the blowup $\frakX^{\tor,\prime}(\underline{v})$ of $\frakX^\tor$ along the ideals $(h_\tau,p^{v_\tau})$. 
Let $\frakX^\tor(\underline{v})$ be a formal model of $\cX^\tor(\underline{v})$, given by the normalization of the greatest open formal subscheme of $\frakX^{\tor,\prime}(\underline{v})$ where the ideal $(h_\tau,p^{v_\tau})$ is generated by $h_\tau$.

Assume that $\underline{n}\leqslant\underline{1}$ and let $(\underline{\widetilde{A}}, (\widetilde{H}_\frakp)_{n_\frakp=1})$ be the semi-abelian scheme over $X_0^\tor(p^{\underline{n}})$ regarded as an $\cO$-scheme. 
  For all $\tau\in \Sigma_p$ such that $n_\frakp = 1$, 
we introduce  a  partial degree function $\deg_\tau$ on $X_0^\tor(p^{\underline{n}})$ as follows (see \cite{Fargues} for a more detailed study of  degree maps for   finite flat commutative group schemes).
The sheaf $\omega_{\widetilde{H}_\frakp}$ of invariant differentials of $\widetilde{H}_\frakp$ is  an  $\frako/\frakp$-module.
As $p$ is unramified in $F$, $\Sigma_\frakp$ is in bijection with the embeddings $\frako/\frakp \hookrightarrow \bar\FF_p$, and hence we have the decomposition $\omega_{\widetilde{H}_\frakp} = \bigoplus_{\tau\in \Sigma_\frakp} \omega_{\widetilde{H}_\frakp,\tau}$.
Then the partial degree function is defined as 
\[\deg_\tau\colon \cX_0^\tor(p^{\underline{n}}) \rightarrow [0,1]\cap \QQ, \quad x = (\underline{A},(H_\frakp)_{n_\frakp=1}) \mapsto \val_p \left(\Fitt_0(\omega_{\widetilde{H}_\frakp,\tau,x})\right).
\]
We note that by \cite[Section 8.3, Lemma 6]{BLR}, any  point  $x = (\underline{A},(H_\frakp)_{n_\frakp=1})$  of $\cX_0^\tor(p^{\underline{n}})$ is defined over the ring of integers of a finite extension of $L$. 

The inverse image of $\deg_\tau$  of a subset of $[0,1]$ defined by a finite number of affine inequalities is an admissible open of $\cX_0^\tor(p^{\underline{n}})$. Moreover, when the inequalities are all non-strict and the coefficients are all rational numbers, then the inverse image is quasi-compact. 

Retain the assumption that $\underline{n}\leqslant \underline{1}$. 
We define  subloci of $\cX_0^\tor(p^{\underline{n}})$ using the partial degree functions as follows. 
Consider the function $(d_\tau)_{\tau\in\Sigma}\colon \cX_0^\tor(p^{\underline{n}}) \rightarrow (\QQ\cap [0,1])^\Sigma$ given by
\[
\begin{cases}
d_\tau = \Hdg_\tau & \text{  if } \tau\in \Sigma_\frakp \text{ with } n_\frakp=0,\\
d_\tau = 1-\deg_\tau & \text{  if } \tau\in \Sigma_\frakp \text{ with } n_\frakp=1.
\end{cases}
\]
For  $\underline{v}\in (\QQ\cap[0,1])^{\Sigma_p}$, define $\cX_0^\tor(p^{\underline{n}},\underline{v})$ as the inverse image of $\displaystyle \prod_{\frakp\in\Sigma_p}(\QQ\cap [0,v_\frakp])$ under $\displaystyle(\sum_{\tau\in\Sigma_\frakp}d_\tau)_{\frakp}$.

Next,  we remove the assumption that $\underline{n}\leqslant \underline{1}$
and define subloci of $\cX_0^\tor(p^{\underline{n}})$.
First note that we have two natural maps of $L$-schemes 
\[
\pi_{1},\pi_{2} \colon X_0^\tor(p^{\underline{n}}) \rightarrow X_0^\tor(p^{\underline{n'}}),\]
where $n'_\frakp=1$ if $n_\frakp\geqslant1$ and $n'_\frakp=0$ if $n_\frakp=0$, given  by 
\begin{align*}
\pi_{1}&\colon (\underline{A}, (H_{\frakp})_{n_\frakp\geqslant1}) \mapsto (\underline{A}, (H_{\frakp}[\frakp])_{n_\frakp\geqslant1}),\\
\pi_{2}&\colon (\underline{A}, (H_{\frakp})_{n_\frakp\geqslant1}) \mapsto (\underline{A'}, (\overline{H}_{\frakp})_{n_\frakp\geqslant1}),
\end{align*}
where $A'$ is the quotient of $A$ by $\prod H_{\frakp}[\frakp^{n_\frakp-1}]$, and $\overline{H}_{\frakp}$ is the image of $H_{\frakp}$ in $A'$.
Then for any $\underline{v} \in (\QQ\cap [0,1])^{\Sigma_p}$, we define $\cX_0^\tor(p^{\underline{n}},\underline{v})$ as the preimage of $\cX_0^\tor(p^{\underline{n'}},\underline{v})$ under $\pi_2$.
The reason why we define $\cX_0^\tor(p^{\underline{n}},\underline{v})$ using $\pi_2$ rather than $\pi_1$ will become apparent in the next section.

We define $\frakX_0^\tor(p^{\underline{n}})$ (resp. $\frakX_0^\tor(p^{\underline{n}},\underline{v})$) as the normalization of $\frakX^\tor$ (resp. $\frakX^\tor(\underline{v})$) in $\cX_0^\tor(p^{\underline{n}})$ (resp. $\cX_0^\tor(p^{\underline{n}},\underline{v})$).

\subsection{Partial canonical subgroups} \

We state a result of Fargues on canonical subgroups. 
Recall that $p$ is odd and $\cO$ is the ring of integers of a finite extension $L$ of $\QQ_p$.
Let $n\geqslant 0$, $h>d\geqslant 0$ be integers. 
Let $G$ be a truncated $p$-divisible group of level $n$, height $h$ and dimension $d$ over $\cO$.
The Hasse invariant of $G$ is $\det[\omega_{G^D} \rightarrow \omega_{G^D}^{(p)}]$,
where $G^D$ denotes the Cartier dual of $G$, and the Hodge height $\Hdg(G)$ of $G$ is defined to be the $p$-adic valuation of a lift of the Hasse invariant. Finally, the Hodge--Tate map $\HT_G\colon G=\Hom(G^D,\Gm) \rightarrow \omega_{G^D}$ is given by $x\mapsto x^\ast \frac{dt}{t}$.

\begin{theorem}[{\cite[TH\'EOR\`EME~6]{Fargue_Can}}] \label{thm:Fargues}
 If $\Hdg(G)<\frac{1}{p^n}$, then the $n$-th step of the Harder--Narasimhan filtration of $G$, denoted by $C_n$ and called the canonical subgroup of level $n$ of $G$,  satisfies the following properties: 
\begin{enumerate}
 \item $C_n(\cO)$ is a $\ZZ/p^n\ZZ$-module free of rank $d$.
 \item $\deg(G/C_n) = \frac{p^n-1}{p-1}\Hdg (G)$. 
 \item For all $1\leqslant n'<n$, $C_{n'}\defeq C_n[p^{n'}]$ is the canonical subgroup of level $n'$ of $G[p^{n'}]$. Moreover, $\Hdg(G/C_{n'}) = p^{n'} \Hdg(G)$ and $C_n/C_{n'}$ is the canonical subgroup of $G/C_{n'}$ of level $n-n'$.
 \item  The orthogonal of $C_n$ with respect to the perfect pairing $G(\cO)\times G^D(\cO) \rightarrow (\ZZ/p^n\ZZ)(1)$ is the canonical subgroup of $G^D$ of level $n$.
 \item For all $1\leqslant n'\leqslant n$, the subgroup 
 $C_{n'}\otimes\cO_{1-p^{n'-1}\Hdg(G)}$ of $G\otimes \cO_{1-p^{n'-1}\Hdg(G)}$ coincides with the kernel of the $n'$-th power of Frobenius.
 \item The group $C_n(\cO)$ coincides with the kernel of the Hodge--Tate map $\HT_{G, n-\frac{p^n-1}{p-1}\Hdg(G)}$. 
 \end{enumerate}
\end{theorem}

Let $A$ be the universal abelian scheme over $X$.  
For  $\frakp\in\Sigma_p$,  $A[\frakp^n]$ is a truncated $p$-divisible group of level $n$, height $2f_\frakp$ and dimension $f_\frakp$. We have $\Hdg(A[\frakp^n])=\sum_{\tau\in\Sigma_{\frakp}} \Hdg_\tau$. 
By  Theorem~\ref{thm:Fargues}  there exists a $\frakp$-canonical subgroup $C_{\frakp,n}\subseteq A[\frakp^n]$ of level $n$ over  $\cX(\underline{v})$, as long as $v_\frakp<\frac1{p^n}$. By functoriality of the canonical subgroup, $C_{\frakp,n}$ is stable under the $\frako$-action. Moreover, according to \cite[\S4.1]{AIP_Siegel}, $C_{\frakp,n}$ extends  to a  
$\frakp$-canonical subgroup $\widetilde{C}_{\frakp,n}$ over $\cX^\tor(\underline{v})$.

The existence of the $\frakp$-canonical subgroup can be rephrased as follows.
Over $\cX^\tor(\underline{v})$ with $v_\frakp<\frac1{p}$, the $\frakp$-canonical subgroup $\widetilde{C}_{\frakp,1}$ of  level $1$ gives rise to a map $\cX^\tor(\underline{v}) \rightarrow \cX_0^\tor(\frakp)$, which is a section to the forgetful map $\cX_0^\tor(\frakp) \rightarrow \cX^\tor$.
By Theorem~\ref{thm:Fargues}(ii), the section identifies $\cX^\tor(\underline{v})$ with $\cX_0^\tor(\frakp,\underline{v})$.
More generally, over $\cX^\tor(\underline{v})$ with $v_\frakp<\frac1{p^n}$, the $\frakp$-canonical subgroup $\widetilde{C}_{\frakp,n}$ of level $n$ gives rise to a map $\cX_0^\tor(\frakp,\underline{v})\xrightarrow{\sim}\cX^\tor(\underline{v}) \rightarrow \cX_0^\tor(\frakp^n)$ which is a section to $\pi_1$, by Theorem~\ref{thm:Fargues}(iii) applied to $n'=1$ .
A computation using Theorem~\ref{thm:Fargues}(ii-iii) with $n'=n-1$ implies that the section identifies $\cX_0^\tor(\frakp,\underline{v})\xrightarrow{\sim}\cX^\tor(\underline{v})$ with $\cX_0^\tor(\frakp^n,\underline{v'})$, where $v'_\frakp = p^{n-1}v_\frakp$ and $v'_\frakq = v_\frakq$ for $\frakq\neq\frakp$.
This is  summarized in the next proposition. 
\begin{proposition} \label{prop:curve_isom}
 Suppose that for all $\frakp\in \Sigma_p$ we have $n_\frakp \in\ZZ_{\geqslant 0}$ and  $v_\frakp < \frac1{p^{n_\frakp}}$.
Then $\cX^\tor(\underline{v})$ is identified with $\cX_0^\tor(p^{\underline{n}},\underline{v'})$, 
where $v_\frakp' = p^{\max(n_\frakp-1,0)}v_\frakp$ for all $\frakp\in \Sigma_p$.
\end{proposition}

\section{Families of partially classical Hilbert modular forms}

\subsection{The modular sheaves}

We fix   $\frakp\in\Sigma_p$ and $n\in\ZZ_{\geqslant 1}$, and we let   $\underline{v}\in (\QQ\cap [0,1])^{\Sigma_p}$ and
$\underline{n}  \in \ZZ_{\geqslant 0}^{\Sigma_p}$  be such that $v_\frakp<\frac1{p^{n}}$ and $n_\frakp=0$. 
 The sheaf $\omega = \omega_{\widetilde{A}}$ on the 
 formal scheme $\frakX_0^\tor(p^{\underline{n}},\underline{v})$ from \S\ref{sec:Hodge} 
 decomposes as a direct sum of invertible sheaves  $ \bigoplus_{\tau\in\Sigma} \omega_{\tau}$ 
according to the $\frako$-action.  We put $\omega_\frakp \defeq \bigoplus_{\tau\in\Sigma_\frakp} \omega_\tau$.
The exact sequence $0 \rightarrow \omega_{\widetilde{A}[p^n]} \rightarrow \omega \xrightarrow{\cdot p^n} \omega\rightarrow 0$ shows that $\omega_{\widetilde{A}[p^n]} = \omega/p^n\omega$,  and furthermore $\omega_{\widetilde{A}[\frakp^n]} = \omega_{\frakp}/p^n\omega_{\frakp}$. 

\begin{proposition} \label{prop:omega_isom}
For $w\leqslant n-\frac{p^n-1}{p-1}v_\frakp$,
the morphism $\omega_{\frakp}  \rightarrow  \omega_{\frakp}/p^n\omega_{\frakp}
=\omega_{\widetilde{A}[\frakp^n]} \to  \omega_{\widetilde{C}_{\frakp,n}}$ of sheaves on $\frakX_0^\tor(p^{\underline{n}},\underline{v})$
induced by the inclusion $\widetilde{C}_{\frakp,n}\subseteq \widetilde{A}[\frakp^n]$
 becomes an isomorphisms modulo $p^w$.
\end{proposition}
\begin{proof}
As in \cite[Prop.~4.2.1]{AIP_Siegel} it suffices to prove the isomorphism at points of $\frakX_0^\tor(p^{\underline{n}},\underline{v})$,
all of which are defined over $\cO_K$ for $K/L$ a finite extension. 

We consider the  exact sequence  $0 \rightarrow \omega_{A[\frakp^n]/C_{\frakp,n}} \rightarrow \omega_{A[\frakp^n]} \rightarrow \omega_{C_{\frakp,n}}\rightarrow 0$ at such a point. Fargues's Theorem~\ref{thm:Fargues}(ii) yields 
\[
 \sum_{\tau\in\Sigma_\frakp} \deg_\tau (A[\frakp^n]/C_{\frakp,n}) = \frac{p^n-1}{p-1}\sum_{\tau\in \Sigma_\frakp}\Hdg_\tau(A[\frakp^n]) \leqslant \frac{p^n-1}{p-1}v_\frakp, 
\]
thus  $p^{\frac{p^n-1}{p-1}v_\frakp}$ annihilates $\omega_{A[\frakp^n]/C_{\frakp,n}}$. Hence 
$\omega_{A[\frakp^n]/C_{\frakp,n}}\subseteq p^{n-\frac{p^n-1}{p-1}v_\frakp}\omega_{A[\frakp^n]} $, implying the claim. 
\end{proof}

The Hodge--Tate map is a morphism of fppf sheaves
\[
\HT_{\frakp,n}\colon \widetilde{C}_{\frakp,n}^{D}\rightarrow \omega_{\widetilde{C}_{\frakp,n}}
\]
sending an $S$-valued point $[x\colon  \widetilde{C}_{\frakp,n,S} \rightarrow \GG_{m,S}]\in \widetilde{C}_{\frakp,n}^{D}(S)$ to $x^\ast (\frac{dt}{t})\in \omega_{\widetilde{C}_{\frakp,n}}(S)$.

Let $\cX_{1,\frakp}^\tor(p^{\underline{n}},\underline{v}) \defeq \Isom_{\cX_0^\tor(p^{\underline{n}},\underline{v})}(\frako/\frakp^n,\widetilde{C}_{\frakp,n}^{D})$ be the $(\frako/\frakp^n)^\times$-torsor over $\cX_0^\tor(p^{\underline{n}},\underline{v})$ classifying trivializations of the partial canonical subgroup $\widetilde{C}_{\frakp,n}$ of level $n$.
Let  $\frakX_{1,\frakp}^\tor(p^{\underline{n}},\underline{v})$ be the normalization of $\frakX_0^\tor(p^{\underline{n}},\underline{v})$ in $\cX_{1,\frakp}^\tor(p^{\underline{n}},\underline{v})$ and let  
\[\psi_n\colon \frako/\frakp^n \xrightarrow{\sim} \widetilde{C}_{\frakp,n}^{D}\] 
be the universal trivialization. Then  $\HT_{\frakp,n}$ induces a linearized Hodge--Tate map
\[
    \HT_{\frakp,n}\otimes 1 \colon \widetilde{C}_{\frakp,n}^{D}\otimes_{\frako} \cO_{\frakX_{1,\frakp}^\tor(p^{\underline{n}},\underline{v})} \rightarrow \omega_{\widetilde{C}_{\frakp,n}}.
\]

\begin{proposition} \label{prop:cokerHT}
    The cokernel of $\HT_{\frakp,n}\otimes 1$
     is annihilated  by $p^\frac{v_\frakp}{p-1}$.
\end{proposition}
\begin{proof}
Following the argument of \cite[Prop.~4.2.2]{AIP_Siegel},
it suffices to prove the isomorphism at points of $\frakX_{1,\frakp}^\tor(p^{\underline{n}},\underline{v})$,
all of which are defined over $\cO_K$ for $K/L$ a finite extension.
As explained in \cite[Prop.~3.2.1]{AIP_Siegel}, the claim can be reduced to the case $n=1$ which is proved in Fargues's Theorem~\ref{thm:Fargues}(vi).
\end{proof}

\begin{theorem}
There exists a sheaf $\cF_\frakp\subseteq \omega_\frakp$  containing $p^{\frac{v_\frakp}{p-1}}\omega_\frakp$ such that
\begin{enumerate}
    \item $\cF_\frakp$ is locally free of rank $f_\frakp$ on $\frakX_{1,\frakp}^\tor(p^{\underline{n}},\underline{v})$, and 
    \item  for all $w\leqslant n-\frac{p^n}{p-1}v_\frakp$ the  map
$\HT_{\frakp,w}\colon \widetilde{C}_{\frakp,n}^{D} \rightarrow \cF_{\frakp}/p^w\cF_{\frakp}$
deduced from $\HT_{\frakp,n}$ induces an isomorphism
\[
\HT_{\frakp,w} \otimes 1 \colon \widetilde{C}_{\frakp,n}^{D}\otimes \cO_{\frakX_{1,\frakp}^\tor(p^{\underline{n}},\underline{v})} \xrightarrow{\sim} \cF_\frakp/p^w\cF_{\frakp}.
\]
\end{enumerate}
\end{theorem}

\begin{proof} Let  $x_1,\ldots, x_{f_{\frakp}}$ be the image of a $\ZZ/p^n\ZZ$-basis of $\frako/\frakp^n$ 
under the universal trivialization $\psi_n\colon \frako/\frakp^n \rightarrow \widetilde{C}_{\frakp,n}^{D}$.
Let $y_1,\ldots, y_{f_\frakp}\in\omega_\frakp$ be lifts of $\HT_{\frakp,n}(x_i)\in \omega_{\widetilde{C}_{\frakp,n}}$ under the surjection $\omega_\frakp\twoheadrightarrow \omega_{\widetilde{C}_{\frakp,n}}$ and let $\cF_\frakp\subseteq \omega_\frakp$ be their linear span. 

We claim that $\cF_\frakp$ is locally free of rank $f_\frakp$. 
Indeed, suppose there is a non-trivial relation $\sum_{i=1}^{f_\frakp} a_i y_i = 0$.
Letting $w_0=n-\frac{p^n-1}{p-1}v_\frakp$, we may assume that there is at least one $a_i$, say $a_1$, not divisible by $p^{w_0-\frac{v_\frakp}{p-1}}$.
In particular, by Proposition~\ref{prop:omega_isom}, after projecting the relation to $\omega_{\widetilde{C}_{\frakp,n},w_0}$, we have $\sum_{i=1}^{f_\frakp} a_i \HT_{\frakp,n}(x_i) = 0$.
By Proposition~\ref{prop:cokerHT}, the kernel of $\HT_{\frakp,n}$ is killed by $p^{\frac{v_\frakp}{p-1}}$, and so $\sum_{i=1}^{f_\frakp} p^{\frac{v_\frakp}{p-1}}a_i x_i=0$. 
However, $p^{\frac{v_\frakp}{p-1}}a_1$ is not divisible by $p^{w_0}$ and hence not divisible by $p^n$, contradicting the fact that the $x_i$'s are images under $\psi_n$ of a basis of $\frako/\frakp^n$.

By Proposition~\ref{prop:cokerHT}, $\cF_\frakp$ contains $p^{\frac{v_\frakp}{p-1}}\omega_\frakp$.
In addition, $\cF_\frakp$ is independent of the choice of the lifts $y_i$.
This is because by the proof of Proposition~\ref{prop:omega_isom}, the kernel of $\omega_\frakp\twoheadrightarrow \omega_{\widetilde{C}_{\frakp.n}}$ is contained in $p^{n-\frac{p^n-1}{p-1}v_\frakp}\omega_\frakp\subseteq p^{\frac{v_\frakp}{p-1}}\omega_\frakp$ and hence contained in $\cF_\frakp$.

Composing $\HT_{\frakp,n}\colon \widetilde{C}_{\frakp,n}^{D} \rightarrow \omega_{\widetilde{C}_{\frakp,n}}$ with reduction modulo $p^{w_0}$ and using Proposition~\ref{prop:omega_isom}, we have the map $\widetilde{C}_{\frakp,n}^{D} \rightarrow \omega_{\frakp}/p^w\omega_\frakp$, which factors through $\cF_\frakp/(p^{w_0}\omega_\frakp \cap \cF_\frakp)$ by construction of $\cF_\frakp$.
Let $w\leqslant n-\frac{p^n}{p-1}v_\frakp$.
Then $p^{w_0}\omega_\frakp \cap \cF_\frakp \subseteq p^w \cF_\frakp$ because $\cF_\frakp$ contains $p^{\frac{v_\frakp}{p-1}}\omega_\frakp$.
Hence we have $\HT_{\frakp,w}\colon \widetilde{C}_{\frakp,n}^{D} \rightarrow \cF_{\frakp}/p^w \cF_{\frakp}$.
Finally, $\HT_{\frakp,w}\otimes 1$ is an isomorphism because it is a surjection between two locally free sheaves of the same rank $f_\frakp$.\end{proof}

For  $w$  such that $n-1<w\leqslant n-\frac{p^n}{p-1}v_\frakp$ we let 
\[
 \gamma_w^0\colon\frakI\frakW_{\frakp,w}^+\rightarrow  \frakX_{1,\frakp}^\tor(p^{\underline{n}},\underline{v})
\]
be the formal scheme classifying $\frako_{\frakp}$-equivariant  trivializations 
\[\psi\colon \cO_{\frakX_{1,\frakp}^\tor(p^{\underline{n}},\underline{v})}\otimes \frako_{\frakp} \xrightarrow{\sim} \cF_\frakp,\] 
 which are $\psi_n$-compatible in the sense that  $\HT_{\frakp,n}\circ\psi_n = \psi \bmod p^w$.

Let $\TT_\frakp = \Res_{\frako_{\frakp}/\ZZ_p}\Gm$, and $\frakT_\frakp$ the formal torus associated to $\TT_\frakp$.
Let $\frakT_{\frakp,w}^0\subseteq \frakT_\frakp$ be the kernel modulo $p^w$, and $\frakT_{\frakp,w}=\frakT_{\frakp,w}^0 \cdot \TT_\frakp(\ZZ_p) \subseteq \frakT_{\frakp}$.
It is easily seen that $\gamma_w^0$ is a formal $\frakT_{\frakp,w}^0$-torsor.
There is an action of $\TT_\frakp(\ZZ_p) = \frako_{\frakp}^\times$ on $\frakX_{1,\frakp}^\tor(p^{\underline{n}},\underline{v})$, where $a\in \frako_{\frakp}^\times$ sends the universal trivialization $\psi_n\colon \frako/\frakp^n\xrightarrow{\sim}\widetilde{C}_{\frakp,n}^{D}$ to $\psi_n(a\cdot -)$.
There is also an action of $\TT_\frakp(\ZZ_p)$ on $\frakI\frakW_{\frakp,w}^+$, where $a\in \frako_{\frakp}^\times$ sends the universal $\psi_n$-compatible trivialization $\psi\colon \cO_{\frakX_{1,\frakp}^\tor(p^{\underline{n}},\underline{v})}\otimes \frako_{\frakp} \xrightarrow{\sim} \cF_\frakp$ to $\psi((1\otimes a)\cdot -)$.
As $\gamma_w^0$ is equivariant for the $\TT_\frakp(\ZZ_p)$-action, the action of $\frakT_{\frakp,w}^0$ on $\gamma_w^0$ extends to an $\frakT_{\frakp,w}$-action   on 
\[\gamma_w\colon \frakI\frakW_{\frakp,w}^+ \xrightarrow{\gamma_w^0}   \frakX_{1,\frakp}^\tor(p^{\underline{n}},\underline{v}) \rightarrow \frakX_0^\tor(p^{\underline{n}},\underline{v}).\]
Since $\frakX_{1,\frakp}^\tor(p^{\underline{n}},\underline{v}) \rightarrow \frakX_0^\tor(p^{\underline{n}},\underline{v})$ is finite \'etale in the generic fiber, $\gamma_w$ becomes a $\cT_{\frakp,w}$-torsor in the generic fiber.

Let $\cW_\frakp$ be the rigid generic fiber of the formal scheme $\Spf(\cO[\![\TT_\frakp(\ZZ_p)]\!])$.
Let $\cV_\frakp\rightarrow \cW_\frakp$ be an affinoid and 
\[
 \kappa_{\cV_\frakp}\colon \TT_\frakp(\ZZ_p) \rightarrow \cO_{\cV_\frakp}^\times
\]
be the universal character.
We say that $\kappa_{\cV_\frakp}$ is \emph{$n$-analytic} for some $n\in\ZZ_{\geqslant 0}$ if the restriction of $\kappa_{\cV_\frakp}$ to $\frakT_n^0(\ZZ_p) = 1+\frakp^n\frako_{\frakp}$ factors as 
\[
  1+\frakp^n\frako_{\frakp} \xrightarrow{\log_{F_\frakp}} \frakp^n\frako_{\frakp} \rightarrow p\cO_{\cV_\frakp} \xrightarrow{\exp} 1+p\cO_{\cV_\frakp}\subseteq \cO_{\cV_\frakp}^\times,
\]
for some $\ZZ_p$-linear map $\frakp^n\frako_{\frakp} \rightarrow p\cO_{\cV_\frakp}$. 
It is known that $\kappa_{\cV_\frakp}$ is $n$-analytic for large enough $n$ depending on $\cV_\frakp$.
Let $\cO_{\cV_\frakp}^\circ\subseteq \cO_{\cV_\frakp}$ be the power-bounded functions on ${\cV_\frakp}$, and $\frakV_\frakp\defeq \Spf\cO_{\cV_\frakp}^\circ$, which is a formal model of ${\cV_\frakp}$. 
When $\kappa_{\cV_\frakp}$ is $n$-analytic, it extends to a character
\[
 \kappa_{\cV_\frakp} \colon \frakT_{\frakp,w} \times \frakV_\frakp \rightarrow \widehat{\GG}_m \times \frakV_\frakp,
\]
for all $w\geqslant n$.
We also say $\kappa_{\cV_\frakp}$ is \emph{$w$-analytic} if such an extension exists.

Let $w$ be such that $n-1<w\leqslant n-\frac{p^n}{p-1}v_\frakp$.
Let $\kappa_{\cV_\frakp}$ be a $w$-analytic character.
Observe that the pushforward $(\gamma_{w}\times 1_{\frakV_\frakp})_\ast(\cO_{\frakI\frakW_{\frakp,w}^+\times \frakV_\frakp})$ is a sheaf on $\frakX_0^\tor(p^{\underline{n}},\underline{v})\times \frakV_\frakp$ equipped with an action of $\frakT_{\frakp,w}$.
We define the $p$-adic modular sheaf on $\frakX_0^\tor(p^{\underline{n}},\underline{v})\times \frakV_\frakp$,
\[
 \omega_{\frakp,w}^{\kappa_{\cV_\frakp}} \defeq (\gamma_{w}\times 1_{\frakV_\frakp})_\ast(\cO_{\frakI\frakW_{\frakp,w}^+\times \frakV_\frakp})[-\kappa_{\cV_\frakp}],
\]
as the subsheaf on which $\frakT_{\frakp,w}$ acts through the character $-\kappa_{\cV_\frakp} \colon \frakT_{\frakp,w} \times \frakV_\frakp \rightarrow \widehat{\GG}_m \times \frakV_\frakp$.
As $\gamma_w$ is the composition of the $\frakT_{\frakp,w}^0$-torsor $\gamma_w^0$ with a finite map, we know that  $\omega_{\frakp,w}^{\kappa_{\cV_\frakp}}$ is a coherent sheaf, which becomes an invertible sheaf when restricted to the generic fiber.

We now discuss the dependence of the modular sheaf on $w$ and $n$. First, for 
 $n-1<w<w'\leqslant n-\frac{p^n}{p-1}v_\frakp$ the natural  $\frakT^0_{\frakp,w'}$-equivariant map from the $\frakT^0_{\frakp,w}$-torsor $\gamma^0_{w}$ to the $\frakT^0_{\frakp,w'}$-torsor $\gamma^0_{w'}$ yields a canonical  isomorphism of $\omega_{\frakp,w}^{\kappa_{\cV_\frakp}}\simeq\omega_{\frakp,w'}^{\kappa_{\cV_\frakp}}$.
Hence $\omega_{\frakp,w}^{\kappa_{\cV_\frakp}}$ only depends on $n=\lceil w\rceil$ and will be henceforth denoted by $\omega_{\frakp,n}^{\kappa_{\cV_\frakp}}$.
In addition, for $n<n'$, and for any $w'$ such that $n'-1<w'\leqslant n'-\frac{p^{n'}}{p-1}v_\frakp$, the
 canonical $\frakT_{\frakp,w'}$-equivariant map from $\gamma_w$ to $\gamma_{w'}$ induces a map $\omega_{\frakp,n}^{\kappa_{\cV_\frakp}} \rightarrow \omega_{\frakp,n'}^{\kappa_{\cV_\frakp}}$.
As $\gamma_w$ is a torsor when restricted to the generic fiber, we have an isomorphism $\omega_{\frakp,n}^{\kappa_{\cV_\frakp}} \xrightarrow{\sim} \omega_{\frakp,n'}^{\kappa_{\cV_\frakp}}$ over the generic fiber.

Let $\TT = \Res_{\frako/\ZZ}\Gm$ and  $\cW$ be the rigid generic fiber of the formal scheme $\Spf(\cO[\![\TT(\ZZ_p)]\!])$. 
Note that $\cW = \prod_{\frakq\in \Sigma_p}\cW_\frakq$.
Let $I\subseteq \Sigma_p$.
Denote by $\overline{I}$ the complement of $I$ in $\Sigma_p$.
Let $\kappa_{\overline{I}} = \prod_{\frakq\in \overline{I}}\kappa_\frakq$, where $\kappa_\frakq$ is a locally algebraic character on $\TT_{\frakq}(\ZZ_p)$
whose algebraic part is of exponent  $(k_\tau)_{\tau\in\Sigma_\frakq}\in \ZZ^{f_\frakq}$ and whose finite part
$\chi_\frakq$ has conductor $n_\frakq$.
Define the partial weight space
\[ \cW_{I}^{\kappa_{\overline{I}}}\defeq \prod_{\frakp\in I} \cW_\frakp \times \kappa_{\overline{I}} \subseteq \cW.
\]
Let $\cV\rightarrow \cW_{I}^{\kappa_{\overline{I}}}$ be an affinoid and $\kappa_{\cV}\colon \TT(\ZZ_p) \rightarrow \cO_{\cV}^\times$
be the resulting universal character. 
As $\TT(\ZZ_p) = \prod_{\frakq\in\Sigma_p} \TT_\frakq(\ZZ_p)$, there exist $\kappa_{\cV,\frakp}$ for all $\frakp\in I$ so that  $\kappa_\cV=\prod_{\frakp\in I}\kappa_{\cV,\frakp} \cdot \kappa_{\overline{I}}$. 
Assume that $\kappa_\cV$ is $w$-analytic, i.e., each $\kappa_{\cV,\frakp}$ extends to a formal character $\kappa_{\cV,\frakp}\colon\frakT_{\frakp,w} \times \frakV\rightarrow \widehat{\GG}_m\times \frakV$.
Let $\underline{n} = (n_\frakq)\in \ZZ_{\geqslant 0}^{\Sigma_p}$ be such that $n_\frakp=0$, for all $\frakp\in I$, and $n_\frakq$ is the conductor of  the finite part $\chi_\frakq$ of $\kappa_\frakq$, for all $\frakq\in\overline{I}$.
According to the above construction, we have the $p$-adic modular sheaf $\omega_{\frakp,w}^{\kappa_{\cV,\frakp}}$ on $\frakX_0^\tor(p^{\underline{n}},\underline{v})\times \frakV$ for all $\frakp\in I$.
Using the notation \eqref{eq:twist}, we define the modular sheaf
\[
 \omega_w^{\kappa_\cV} \defeq \bigotimes_{\frakp\in I} \omega_{\frakp,w}^{\kappa_{\cV,\frakp}}  \otimes \bigotimes_{\frakq \in \overline{I},\tau\in\Sigma_\frakq} \omega_\tau^{k_\tau}(\chi_\frakq)
\]
on $\frakX_0^\tor(p^{\underline{n}},\underline{v})\times \frakV$. Restricted to the generic fiber $\cX_0^\tor(p^{\underline{n}},\underline{v})$, the sheaf is independent of  the choice of $w$ and will be denoted by  $\omega^{\kappa_\cV}$. The proof of the following  base change property is standard. 
\begin{proposition} \label{prop:sheafbasechange}
Let $f\colon \cV'\rightarrow \cV$ be a morphism of affinoid rigid spaces.
Then over the generic fiber $\cX_0^\tor(p^{\underline{n}},\underline{v})$, we have an isomorphism
$ f^\ast \omega^{\kappa_\cV}\isom \omega^{\kappa_{\cV'}}$. 
\end{proposition}

Now we perform descent from $G^\ast$ to $G$.
We have defined $\cW$ as the rigid generic fiber of the formal scheme $\Spf(\cO[\![\TT(\ZZ_p)]\!])$.
It is the weight space for $G^\ast$.
Let $\cW^G$ be the rigid generic fiber of $\Spf(\cO[\![\TT(\ZZ_p)\times\ZZ_p^\times]\!])$, which is the weight space for $G$.
Consider the map of weight spaces 
\[\cW^G\rightarrow\cW, \quad (\nu,w)\mapsto  \nu^2 \cdot (w\circ\Nm_{F/\QQ}), \] 
 extending the map on classical weights defined in \S\ref{HMV}.
Let $\kappa_{\overline{I}}$ be a locally algebraic character on $\prod_{\frakq\in \overline{I}} \TT_{\frakq}(\ZZ_p)$ and 
 $\cW_{I}^{G,\kappa_{\overline{I}}}$ be the preimage of $\cW_{I}^{\kappa_{\overline{I}}}$ under $\cW^G \rightarrow \cW$.
Let $\cU\rightarrow\cW_{I}^{G,\kappa_{\overline{I}}}$ be an affinoid and
\[
(\nu_\cU,w_\cU)\colon \TT(\ZZ_p)\times\ZZ_p^\times \rightarrow \cO_\cU^\times
\]
be the universal character, which we assume to be $w$-analytic.
By composing with $\cW_{I}^{G,\kappa_{\overline{I}}} \rightarrow \cW_{I}^{\kappa_{\overline{I}}}$, we obtain another $w$-analytic
character 
\[
\kappa_\cU=\nu_\cU^2 \cdot (w_\cU\circ \Nm_{F/\QQ}) \colon   \TT(\ZZ_p) \rightarrow \cO_\cU^\times.
\]

Note that the construction of modular sheaves in this section is done in fact over $X_\frakc$ for each fractional ideal $\frakc$ of $F$, so we add the subscript $\frakc$ to the modular sheaves.
Exactly as for the 
classical modular sheaves one performs a twisted 
$\Delta_\frakn$-descent to obtain modular sheaves for $G$:
\[
\omega_{G,\frakc}^{(\nu_\cU,w_\cU)} \defeq \left(\omega_\frakc^{\kappa_\cU}(-\nu_\cU)\right)^{\Delta_{\frakn}}.
\]

\subsection{Families of partially classical cuspforms}
Fix  $I\subseteq \Sigma_p$ and let $\kappa_{\overline{I}}$ be a locally algebraic character on $\prod_{\frakq\in \overline{I}} \TT_{\frakq}(\ZZ_p)$.
Let $\cV\rightarrow \cW_{I}^{\kappa_{\overline{I}}}$ be an affinoid and $\kappa_{\cV}\colon \TT(\ZZ_p) \rightarrow \cO_{\cV}^\times$ be the universal character.
Let $\underline{n} = (n_\frakq)\in \ZZ_{\geqslant 0}^{\Sigma_p}$ be such that $n_\frakp=0$ for all $\frakp\in I$ and $n_\frakq$ is the conductor of the finite part of $\kappa_\frakq$ for all $\frakq\in\overline{I}$.
Let $\underline{v}\in(\QQ\cap[0,1])^{\Sigma_p}$ be such that $ v_\frakp<\frac{1}{p}$  for all $\frakp\in I$ and  $v_\frakq = 1$ for all $\frakq\in\overline{I}$.
Consider the complex of  $\overline{I}$-classical coherent cuspidal cohomologies of Hilbert modular varieties parametrized by $\cV$,
\[
  \RG^{\overline{I}}(\frakc,p^{\underline{n}}, \kappa_\cV,\underline{v})
\defeq 
\RG(\cX_0^\tor(\frakc,p^{\underline{n}},\underline{v})\times \cV, \omega^{\kappa_\cV}(-D)),
\] 
where $D$ is the cuspidal divisor of the toroidal compactification.

The base change property from Proposition~\ref{prop:sheafbasechange} induces specialization maps:
\[
  \RG^{\overline{I}}(\frakc,p^{\underline{n}}, \kappa_\cV,\underline{v})
  \rightarrow   \RG^{\overline{I}}(\frakc,p^{\underline{n}}, \kappa_{\cV'},\underline{v}),
\]
for any morphism of affinoid rigid spaces $\cV' \rightarrow \cV$.
When $\cV' \rightarrow \cV$ is flat, the specialization map is an isomorphism.

Turning our attention from $G^\ast$ to $G$, define
\[
  \RG_G^{\overline{I}}(\frakc,p^{\underline{n}}, (\nu_\cU,w_\cU),\underline{v})
\defeq 
\RG(\cX_0^\tor(\frakc,p^{\underline{n}},\underline{v})/{\Delta_\frakn}\times \cU, \omega_{G,\frakc}^{(\nu_\cU,w_\cU)}(-D)).
\] 

As in the classical case, for $x\in F_+^{\times}$, the isomorphism $L_{\frakc,x}\colon X_{\frakc} \rightarrow X_{x\frakc}$ induces an isomorphism
\[
L_{\frakc,x}^\ast \colon \RG_G^{\overline{I}}(x\frakc,p^{\underline{n}}, (\nu_\cU,w_\cU),\underline{v}) \rightarrow \RG_G^{\overline{I}}(\frakc,p^{\underline{n}},(\nu_\cU,w_\cU),\underline{v}).
\]
Note that $L_{\frakc,x}^\ast$ depends only on the fractional ideal $(x)$ as long as its generator $x$ is chosen as a $p$-adic unit.
Let $\Frac(F)^{(p)}$ be the group of fractional ideals of $F$ prime to $p$.
Let $F_+^{\times,(p)}$ be the group of totally positive elements of $F$ which are $p$-adic units. 
Note that $\Frac(F)^{(p)}/\{(x)\colon x\in F_+^{\times,(p)}\}$ is equal to the strict class group $\Cl$ of $F$.
We then define
\[
\RG_G^{\overline{I}}(p^{\underline{n}},  (\nu_\cU,w_\cU),\underline{v})\defeq
\bigoplus_{\frakc\in \Frac(F)^{(p)}}\RG_G^{\overline{I}}(\frakc,p^{\underline{n}}, (\nu_\cU,w_\cU),\underline{v})/\left(L_{\frakc,x}^\ast - \id\right)_{x\in F_+^{\times,(p)}}.
\]

We also define
\[
   \RG^{\overline{I}}(\frakc, p^{\underline{n}}, \kappa_\cV) \defeq \varinjlim_{v_\frakp\rightarrow 0^+, \frakp\in I}    \RG^{\overline{I}}(\frakc,p^{\underline{n}}, \kappa_\cV,\underline{v}),
\]
\begin{equation}\label{G-families}
  \RG_G^{\overline{I}}( p^{\underline{n}}, (\nu_\cU,w_\cU)) \defeq \varinjlim_{v_\frakp\rightarrow 0^+, \frakp\in I}   \RG_G^{\overline{I}}( p^{\underline{n}}, (\nu_\cU,w_\cU),\underline{v}).
\end{equation}

\subsection{Functoriality}
Let $I\subseteq I' \subseteq \Sigma_p$.
Fix locally algebraic characters $\kappa_{\overline{I'}}$ on $\prod_{\frakq \in \overline{I'} } \frako_{\frakq}^\times$ and $\kappa_{I'\setminus I}$ on $\prod_{\frakq\in I'\setminus I} \frako_{\frakq}^\times$, and let $\kappa_{\overline{I}} = \kappa_{\overline{I'}}\cdot \kappa_{I'\setminus I}$.
Let $\iota_{I'\setminus I}$ be the inclusion of weight spaces $\cW_{I}^{\kappa_{\overline{I}}} \subseteq \cW_{I'}^{\kappa_{\overline{I'}}}$.
Let $\iota_\cV \colon \cV\rightarrow \cW_{I}^{\kappa_{\overline{I}}}$ be an affinoid and $\kappa_{\cV}\colon \TT(\ZZ_p) \rightarrow \cO_{\cV}^\times$ be the universal character.
We then have $\iota_{I'\setminus I}\circ \iota_\cV\colon \cV \rightarrow \cW_{I'}^{\kappa_{\overline{I'}}}$ whose universal character is again denoted by $\kappa_\cV$.

Let $\underline{n} = (n_\frakq), \underline{n}' = (n'_\frakq)\in \ZZ_{\geqslant 0}^{\Sigma_p}$ be such that $n_\frakp=0$ for all $\frakp\in I$, $n_\frakq$ is the conductor of the finite part of $\kappa_\frakq$ for all $\frakq\in\overline{I}$, $n'_\frakp=0$ for all $\frakp\in I'$, and $n'_\frakq$ is the conductor of the finite part of $\kappa_\frakq$ for all $\frakq\in\overline{I}$.
Let $\underline{v},\underline{v}'\in(\QQ\cap[0,1])^{\Sigma_p}$ be such that $v_\frakq = 1$ for all $\frakq\in\overline{I}$, $v'_\frakq = 1$ for all $\frakq\in\overline{I'}$.
Also, assume that $v'_\frakp \leqslant v_\frakp<\frac{1}{p}$ for all $\frakp\in I$ and $v'_\frakq <\frac1{p^{n_\frakq}}$ for all $\frakq\in I'\setminus I$.

Let $\underline{v}'' \in (\QQ\cap [0,1])^{\Sigma_p}$ be such that $v''_\frakq = p^{\max(n_\frakq-1,0)}v'_\frakq$ for all $\frakq\in I'\setminus I$, $v''_\frakq = v'_\frakq = v_\frakq = 1$ for all $\frakq \in \overline{I'}$, and $v''_\frakq = v'_\frakp$ for all $\frakp\in I$.
According to the theory of canonical subgroups (see Proposition~\ref{prop:curve_isom}),  
\[
\cX_0^\tor(\frakc,p^{\underline{n}'},\underline{v}') \isom \cX_0^\tor(\frakc,p^{\underline{n}},\underline{v}'')
\subseteq \cX_0^\tor(\frakc,p^{\underline{n}},\underline{v}).
\]
Hence there is a natural $\Delta_\frakn$-equivariant map,
\begin{align*} 
\RG^{\overline{I}}(\frakc,p^{\underline{n}},\kappa_\cV,\underline{v}) \rightarrow \RG^{\overline{I'}}(\frakc,p^{\underline{n}},\kappa_\cV,\underline{v'}),
\end{align*}
which after taking direct limits induces
\begin{align} \label{eq:functoriality}
\RG^{\overline{I}}(\frakc,p^{\underline{n}},\kappa_\cV) \rightarrow \RG^{\overline{I'}}(\frakc,p^{\underline{n}},\kappa_\cV).
\end{align}

\section{Eigenvariety of partially classical Hilbert modular forms}
\subsection{Hecke operators}
Let $I\subseteq \Sigma_p$. 
Let $\frakl$ be a prime ideal of $\frako$ which is either in $I$, or relatively prime to  $\frakn p$. 
In this section,  we will define the Hecke operators $T_\frakl$ for $\frakl\nmid \frakn p $ and $U_\frakp$ for $\frakl=\frakp\in I$. 
We recall that  $\underline{n} = (n_\frakq) \in \ZZ_{\geqslant 0}^{\Sigma_p}$ is  such that $n_\frakp=0$ for all $\frakp\in I$.
Let  $\underline{v}\in (\QQ\cap [0,1])^{\Sigma_p}$  be such that  $v_\frakp <\frac{1}{p}$ for all  $\frakp\in I $. 
Let   $v'_\frakp = v_\frakp/p$, if  $\frakl=\frakp\in I $, and $v'_\frakq = v_\frakq$, in all other cases.

Define $\cX_0(\frakl,\frakc,p^{\underline{n}},\underline{v})$ to be the moduli space classifying  $(\underline{A},(H_{\frakq})_{n_\frakq\geqslant 1}, H)$, where $(\underline{A}, (H_{\frakq})_{n_\frakq\geqslant 1})\in \cX_0(\frakc, p^{\underline{n}},\underline{v})$ and $H\subseteq A[\frakl]$ is a finite flat isotropic $\frako$-subgroup scheme, which is \'etale-locally isomorphic to $\frako/\frakl$; in addition, if $\frakl=\frakp\in I$, we assume that $H$ is not the  $\frakp$-canonical subgroup  $C_\frakp$ of level $1$, which exists by the assumption that $v_\frakp<\frac{1}{p}$.

We  have a correspondence 
\begin{center}
  \begin{tikzcd}
   &\cX_0(\frakl,\frakc,p^{\underline{n}},\underline{v}) \ar[ld, "p_1"'] \ar[rd, "p_2"] & \\
  \cX_0(\frakc,p^{\underline{n}},\underline{v})&& \cX_0(\frakl\frakc,p^{\underline{n}},\underline{v'}),
  \end{tikzcd}
\end{center}
given by 
\begin{align*}
p_1&\colon  (\underline{A},(H_{\frakq})_{n_\frakq\geqslant 1},H) \mapsto (\underline{A}, (H_{\frakq})_{n_\frakq\geqslant 1}),  \\
p_2&\colon  (\underline{A},(H_{\frakq})_{n_\frakq\geqslant 1},H) \mapsto (\underline{A}/H, (\overline{H}_\frakq)_{n_\frakq\geqslant 1}), 
\end{align*}
where $\overline{H}_\frakq$ is the image of $H_\frakq$ under $A \rightarrow A/H$. 

\begin{proposition} \label{prop:Hecke} \hfill
\begin{enumerate}
\item When $\frakl\nmid \frakn p$, then $p_1$ and $p_2$ are both finite \'etale of degree $1+q_\frakl$. 
\item When $\frakl=\frakp\in I$,  then $p_1$  is \'etale of degree $q_\frakp$, whereas $p_2$ is an isomorphism.
\end{enumerate}
\end{proposition}
\begin{proof} \hfill
\begin{enumerate}
\item When $\frakl\nmid \frakn p$, then $A$ and $A'=A/H$  share the same  partial Hodge heights and partial degrees. 
Hence the statements about $p_1$ and $p_2$ are clear.

\item When $\frakl=\frakp\in I$, then as in (i), $A$ and $A'=A/H$ share the same  partial Hodge heights and partial degrees, except at $\frakp$. The statement about $p_1$ is still clear. To see $p_2$ is an isomorphism, first suppose that $(\underline{A},(H_{\frakq})_{n_\frakq\geqslant 1}, H)\in \cX_0(\frakl,\frakc,p^{\underline{n}},\underline{v})$, and let $C_\frakp$ be the $\frakp$-canonical subgroup of level $1$ of $A$. 
By  \cite[Cor.~3]{Fargues} we have $\deg_\tau ((C_\frakp+H)/H) \geqslant \deg_\tau (C_\frakp)$.
It then follows from 
 \cite[Prop.~3.1.2]{AIP_Siegel} that $C'_\frakp:=(C_\frakp+H)/H=A[\frakp]/H$ is the canonical subgroup of level $1$ of $A'=A/H$. 
 Hence $A'/C'_\frakp=A/A[\frakp]$ and Theorem~\ref{thm:Fargues}(iii) yields
 \begin{align*}
p \sum_{\tau\in\Sigma_\frakp}\Hdg_\tau(A') = \sum_{\tau\in\Sigma_\frakp}\Hdg_\tau(A'/C'_\frakp)= \sum_{\tau\in\Sigma_\frakp}\Hdg_\tau(A),
\end{align*}
implying that  $p_2(\underline{A},(H_{\frakq})_{n_\frakq\geqslant 1}, H) \in \cX_0(\frakl\frakc,p^{\underline{n}},\underline{v'})$.
We also see from the argument that the $p_2$ is an isomorphism, as we can recover $A$ through $A=(A/A[\frakp]) \otimes_{\frako} \frakp=(A'/C'_\frakp) \otimes_{\frako} \frakp$.
\end{enumerate}
\qedhere

\end{proof}

The universal isogeny on $\cX_0(\frakl,\frakc,p^{\underline{n}},\underline{v})$ yields a natural morphism of  sheaves on $\cX_0(\frakl,\frakc,p^{\underline{n}},\underline{v})\times \cV$:
\[ \pi_\frakl\colon p_1^\ast \omega^{\kappa_\cV}_{\frakl\frakc}\xrightarrow\sim p_2^\ast \omega^{\kappa_\cV}_{\frakc}. \]
We then define the Hecke operators $T_\frakl$ for $\frakl \nmid \frakn p$ (resp. $U_\frakp$ for $\frakl=\frakp\in I$) as the composed map
\begin{align*}
\RG^{\overline{I}}(\frakl\frakc, p^{\underline{n}}, \kappa_\cV,\underline{v})&\xrightarrow{\mathrm{res}} 
\RG^{\overline{I}}(\frakl\frakc, p^{\underline{n}}, \kappa_\cV,\underline{v}')
=\RG(\cX_0(\frakl\frakc,p^{\underline{n}},\underline{v'})\times \cV, \omega^{\kappa_\cV}_{\frakl\frakc}(-D))\\
&\xrightarrow{p_2^\ast}  
\RG(\cX_0(\frakl,\frakc,p^{\underline{n}},\underline{v})\times \cV, p_2^\ast\omega^{\kappa_\cV}_{\frakl\frakc}(-D)) \\
&\xrightarrow{(\pi_\frakl^{\ast})^{-1}}
\RG(\cX_0(\frakl,\frakc,p^{\underline{n}},\underline{v})\times \cV, p_1^\ast\omega^{\kappa_\cV}_\frakc(-D)) \\
&\xrightarrow{{\frac1{q_\frakl}}\Tr_{p_1}} \RG^{\overline{I}}(\frakc, p^{\underline{n}}, \kappa_\cV,\underline{v}).
\end{align*}

Let $(A,A^+)$ be a complete Tate algebra over a non-Archimedean field.
We let $\varpi\in A^+$ be a pseudo-uniformizer.
A Banach $A$-module $M$ is a topological $A$-module whose topology can be described as follows: Let $A_0$ be an open and bounded subring of $A$.
Then $M$ contains an open and bounded sub $A_0$-module $M_0$ which is $\varpi$-adically complete and separated (and $M = M_0[1/\varpi]$).
Recall that a morphism $M\rightarrow N$ of Banach $A$-modules is called compact if it is a limit of finite rank operators for the supremum norm of operators.
A morphism $M^\bullet \rightarrow N^\bullet$ of complexes of Banach $A$-modules is called compact if it is compact in each degree.

\begin{lemma} \label{lem:Upcompact}
The operator $U_I=\prod_{\frakp\in I}U_\frakp$ is compact.
\end{lemma}
\begin{proof}
We write $I\frakc$ for the fractional ideal $(\prod_{\frakp\in I} \frakp) \frakc$.
By Proposition~\ref{prop:Hecke}, $U_I$ factors through the restriction of domain
\begin{align*}
\RG(\cX_0(I\frakc,p^{\underline{n}},\underline{v})\times \cV, \omega^{\kappa_\cV}_{I\frakc}(-D))
\rightarrow 
\RG(\cX_0(I\frakc,p^{\underline{n}},\underline{v'})\times \cV, \omega^{\kappa_\cV}_{I\frakc}(-D)),
\end{align*}
where $v'_\frakp = v_\frakp/p$ for $\frakp\in I$ and $v'_\frakq = v_\frakq$ for $\frakq\notin I$. Since the closure of $\cX_0(I\frakc,p^{\underline{n}},\underline{v'})$ is contained in $\cX_0(I\frakc,p^{\underline{n}},\underline{v})$, the restriction map is compact by \cite[Lemma~2.5.23]{BP}.
Hence $U_I$ is also compact.
\end{proof}

Now we show how the Hecke operators for $G^\ast$ yield  Hecke endomorphisms for $G$.

The Hecke operator $T_\frakl$  induces an endomorphism of $\RG_G^{\overline{I}}(p^{\underline{n}},(\nu_\cU,w_\cU),\underline{v})$
defined in \eqref{G-families}. 
To define an  endomorphism $U_\frakp$, we need a modification.
For a prime $\frakp\in \Sigma_p$ we choose $x_\frakp\in F_+^{\times}$ such that $\val_\frakp(x_{\frakp})=1$ and $\val_\frakq(x_{\frakp})=0$ for all $\frakq\mid p$ but $\frakq\neq \frakp$.
We can then identify $\RG_G^{\overline{I}}(\frakp\frakc,p^{\underline{n}},(\nu_\cU,w_\cU),\underline{v})$ with $\RG_G^{\overline{I}}(x_{\frakp}^{-1}\frakp\frakc,p^{\underline{n}},(\nu_\cU,w_\cU),\underline{v})$ allowing us to define an endomorphism $U_\frakp$ of 
$\RG_G^{\overline{I}}(p^{\underline{n}},(\nu_\cU,w_\cU),\underline{v})$, which nevertheless depends on  the choice of $x_\frakp$.

\subsection{Admissible covering of the partial ordinary locus}
Let $B$ denote the standard Borel subgroup of $G$ over $\QQ_p$.  The parabolic subgroups of $G$ containing $B$ are exactly the 
\[P_I=\prod_{\frakp\in I} B(F_\frakp)\cdot \prod_{\frakp\in \overline{I}} \GL_2(F_\frakp),\]
where $I$ runs over the subsets of  $\Sigma_p$. Note that $P_{\varnothing}= G$ while $P_{\Sigma_p}= B$. 
The Weyl group $W_I$ of the Levi factor of  $P_I$ can be identified with $\{1,w_0\}^{\overline{I}}$, where $w_0$ denotes the non-trivial Weyl element for $\mathrm{GL}(2)$, and hence the set ${}^IW$ of minimal length coset representatives of  $W_I\backslash W$ can be identified with $\{1,w_0\}^{I}$. For $J\subseteq \Sigma_p$, define $w_J\in W_G$ as the element which is $w_0$ at exactly the factors corresponding to $J$. The length of $w_J$ equals $|J|$. 
Then $W_I=\{w_J\colon J\subseteq \overline{I}\}$ and ${}^I W = \{w_J\colon J\subseteq I\}$.

Recall the Bruhat decomposition of the flag variety 
\[\mathrm{FL} =B\backslash G=\coprod_{J\subseteq \Sigma_p}C_J\]
into cells $C_J\defeq B\backslash B w_J B$. 
Let $S_I$ be the Schubert variety which is the Zariski closure of $C_I$, namely $S_I = \coprod_{J\subseteq I} C_J$.
Note that   $C_J\isom (\AA^1)^J$ and $S_J \isom (\PP^1)^J$.
More precisely, consider the root system $\Phi=\coprod_{\frakp\in\Sigma_p}\{\alpha_\frakp, -\alpha_\frakp\}$ of $G$, where $\alpha_\frakp$ is the positive root corresponding to the $\frakp$-copy of $\GL_2$. 
The set of positive roots is $\Phi^+ = \{\alpha_\frakp\colon \frakp\in\Sigma_p\}$ and the set of negative roots is $\Phi^- = \{-\alpha_\frakp\colon \frakp\in\Sigma_p\}$.
For each root $\alpha$ of $G$, let $U_\alpha \isom \AA^1$  be the corresponding one parameter subgroup.
Then there is an isomorphism of schemes (see, for example,  \cite[Lem.~3.1.3]{BP})
\[
\prod_{\alpha\in w_I^{-1} \Phi^- \cap \Phi^+} {U_\alpha} = \prod_{\frakp\in I} {U_{\alpha_\frakp}} \xrightarrow\sim C_{I}, \quad (u_\alpha) \mapsto w_I\prod u_\alpha.
\]

Let $T_I \defeq C_\Sigma \cdot w_\Sigma^{-1}w_I$ be the twisted open cell, which contains $C_I$.
Then we also have 
\[
\prod_{\frakp\in I} U_{\alpha_\frakp}\times \prod_{\frakq\in \overline{I}} {U_{-\alpha_\frakq}} \xrightarrow\sim T_{I}, \quad (u_\alpha) \mapsto w_I\prod u_\alpha.
\]
 
Let $\overline{C}_I$ be the reduction of $C_I$ modulo $p$.
Let $]\overline{C}_I[\subseteq \cF\cL$ be the tube.
Since $C_I\isom (\AA^1)^{I}$, $]\overline{C}_I[$ is isomorphic to the product of $|I|$ closed unit balls and $|\overline{I}|$ open unit balls.
More precisely, for each $\alpha\in\Phi$,
let $\cU_\alpha$ be the rigid generic fiber associated with the integral model of $U_\alpha$; namely $\cU_\alpha$ is isomorphic to the closed unit ball.
Let $\cU_\alpha^\circ$ be the tube of the identity element in $\cU_\alpha$, which is isomorphic to the open unit ball. 
By \cite[Cor.~3.3.5]{BP} there is an isomorphism of analytic spaces
\[
\prod_{\alpha\in w_I^{-1} \Phi^- \cap \Phi^+} {\cU_\alpha} \times \prod_{\alpha\in w_I^{-1} \Phi^- \cap \Phi^-} \cU_\alpha^\circ
=\prod_{\frakp\in I} {\cU_{\alpha_\frakp}} \times \prod_{\frakq\in \overline{I}} \cU_{-\alpha_\frakq}^\circ
\xrightarrow\sim ]\overline{C}_{I}[, \quad (u_\alpha) \mapsto w_I\prod u_\alpha.
\]
Let $\cU_\alpha^\an$ be the rigid analytic space associated to $U_\alpha$, namely $\cU_\alpha^\an$ is isomorphic to the analytic affine line. Then $\cT_I^\an$ contains $]\overline{C}_{I}[$ and 
\[
\prod_{\frakp\in I} \cU_{\alpha_\frakp}^\an \times \prod_{\frakq\in \overline{I}} \cU_{-\alpha_\frakq}^\an  \xrightarrow\sim \cT_I^\an, (u_\alpha) \mapsto w_I\prod u_\alpha.
\]

For $\ell\in\ZZ$ and $\alpha\in\Phi$, let $\cU_{\alpha,\ell}\subseteq \cU_{\alpha}^\an$ be the closed ball of radius $p^{-\ell}$, so $\cU_{\alpha,0} = \cU_\alpha$.
Let $J\subseteq I$ be two subsets of $\Sigma_p$.
Define $]\overline{C}_J[_{I,\ell,m}$ to be the image of 
\[\prod_{\frakp\in J} \cU_{\alpha_\frakp,\ell} \times \prod_{\frakq\in I\setminus J} \cU_{-\alpha_\frakq,\ell} \times \prod_{\frakq\in \overline{I}} \cU_{-\alpha_\frakq,m}\]
in $\cT_J^{\an}$.
For example, we have $]\overline{C}_{\Sigma_p}[_{\Sigma_p,0,0} = ]\overline{C}_{\Sigma_p}[$ and $]\overline{C}_J[_{I,0,0} \supseteq ]\overline{C}_{J}[$.
We observe that for any $\ell<0<m$,  $\{]\overline{C}_J[_{I,\ell,m}\}_{J\subseteq I}$ form an admissible affinoid covering of $]\overline{S}_I[_m$, a strict neighborhood of $\cS_I$.

Let $X_\infty\rightarrow X$ be the Hilbert modular variety of $p^\infty$-level. 
Choosing  a toroidal compactification $X_\infty^\tor\rightarrow X^\tor$ and the minimal compactification $X_\infty^\ast \rightarrow X^\ast$, we have the Hodge--Tate period morphisms
\begin{center}
\begin{tikzcd}
  \cX_\infty^\tor \ar[d, "\pi"'] \ar[rd, "\pi_{HT}^\tor"] & \\
\cX_\infty^\ast \ar[r,"\pi_{HT}"]& \cF\cL
\end{tikzcd}
\end{center}
 between perfectoid spaces. The two morphisms $\pi_{\HT}$ and $\pi_{\HT}^\tor$ are Hecke equivariant, and agree on the open subset $\cX_\infty^{\an}$. Moreover, $\pi_\HT$ is affine.
Note that $\pi_{\HT}^{-1}(\cC_{\varnothing})$ is the canonical ordinary locus $\cX^\ast_\infty(\underline{0})$. 
More generally, $\pi_{\HT}^{-1}(\cS_{\overline{I}})$ is the $I$-canonical $I$-ordinary locus $\cX_I^\ast \defeq \cX_\infty^\ast(\underline{v})$, where $v_\frakp=0$ for $\frakp\in I$ and $v_\frakq=1$ for $\frakq\in \overline{I}$.

\begin{proposition}
For any $\ell<0<m$,  $\{ \pi_{\HT}^{-1}( ]\overline{C}_{\overline{J}}[_{\overline{I},\ell,m}) \colon J\supseteq I \}$ is an admissible affinoid covering of $\pi_{\HT}^{-1}(]\overline{S}_{\overline{I}}[_m)$, a strict neighborhood of the $I$-canonical $I$-ordinary locus $\cX_I^\ast$.
\end{proposition}
\begin{proof}
This is due to the fact that $\pi_\HT$ is affine and the observation in the paragraphs above that for any $\ell<0<m$, $\{]\overline{C}_{\overline{J}}[_{\overline{I},\ell,m}\}_{J\supseteq I}$ form an admissible affinoid covering of $]\overline{S}_{\overline{I}}[_m$, a strict neighborhood of $\cS_{\overline{I}}$.
\end{proof}

Let $\underline{n}\in \ZZ_{\geq0}^{\Sigma_p}$. 
Taking quotient by $K_0(p^{\underline{n}}) \defeq \prod_{\frakp\in\Sigma_p} K_0(\frakp^{n_\frakp}) 
\subseteq\prod_{\frakp\in \Sigma_p}\GL_2(\frako_{\frakp})$, we have the truncated Hodge--Tate period map
\[
\pi_{\HT, K_0(p^{\underline{n}})}\colon \cX_0^\ast(\frakc,p^{\underline{n}}) \rightarrow   \cF\cL/K_0(p^{\underline{n}}),
\]
which is a continuous Hecke-equivariant map between topological spaces.

Let $I\subseteq \Sigma_p$ and $m\in\ZZ_{\geqslant 1}$.
Assume that $\underline{n}$ is such that $n_\frakp = m$ for all $\frakp\in I$.
Then for any $J\supseteq I$, $]\overline{C}_{\overline{J}}[_{\overline{I},\ell,m}$ is invariant under $K_0(p^{\underline{n}})$. 
For $\ell\in \ZZ$ the space 
\[\cY_{\overline{J},\ell,m} \defeq \pi^{-1}_{\HT, K_0(p^{\underline{n}})}(]\overline{C}_{\overline{J}}[_{\overline{I},\ell,m}K_0(p^{\underline{n}}))\] 
is an affinoid open of $\cX_0^\ast(p^{\underline{n}})$.
If $\ell<0$, we then have an admissible affinoid covering $\{\cY_{\overline{J},\ell,m}\colon J\supseteq I\}$ of $\pi_{\HT, K_0(p^{\underline{n}})}^{-1}(]\overline{S}_{\overline{I}}[_m K_0(p^{\underline{n}}))$ and  it is easy to see 
from the definition that  the intersection of any number of them
is still an affinoid
\[
\cY_{\overline{J}_1,\ldots,\overline{J}_r,\ell,m} \defeq
\cY_{\overline{J}_1,\ell,m}\cap \cdots \cap \cY_{\overline{J}_r,\ell,m}.
\]

We also have the map 
\[ \pi_{\HT,K_0(p^{\underline{n}})}^\tor \colon \cX^\tor_0(\frakc,p^{\underline{n}}) \rightarrow \cX_0^\ast(\frakc,p^{\underline{n}}) \xrightarrow{\pi_{\HT,K_0(p^{\underline{n}})}} \cF\cL/K_0(p^{\underline{n}}),\]
and we consider  $\cY_{\overline{J},\ell,m}^{\tor} \defeq \pi^{\tor,-1}_{\HT, K_0(p^{\underline{n}})}(]\overline{C}_{\overline{J}}[_{\overline{I},\ell,m}K_0(p^{\underline{n}}))$
and more generally $\cY^{\tor}_{\overline{J}_1,\ldots,\overline{J}_r,\ell,m}$.

Let $\cV\rightarrow \cW_{I}^{\kappa_{\overline{I}}}\defeq \prod_{\frakp\in I} \cW_\frakp \times \kappa_{\overline{I}} \subseteq \cW$ be an affinoid such that the universal character $\kappa_{\cV}$ is $m$-analytic and 
consider $\RG(\pi_{\HT, K_0(p^{\underline{n}})}^{\tor, -1}(]\overline{S}_{\overline{I}}[_m K_0(p^{\underline{n}}))\times \cV, \omega^{\kappa_\cV}(-D))$.

\begin{lemma} \label{lem:projective}
For  $\cY^{\tor}=\cY^{\tor}_{\overline{J}_1,\ldots,\overline{J}_r,\ell,m}$, 
the Banach $\cO(\cV)$-module 
$\rH^0(\cY^{\tor}\times \cV, \omega_m^{\kappa_\cV}(-D))$ is  projective. 
\end{lemma}
\begin{proof}
As the image $\cY$ of $\cY^{\tor}$ in $\cX_0^\ast(\frakc,p^{\underline{n}},\underline{v})$ is an affinoid, a delicate argument of Andreatta, Iovita, and Pilloni \cite[\S3.6]{AIP16}  using the explicit description of the formal completion along the cuspidal divisor shows that 
\[R^q\left(\pi\times \mathbbm{1}_{\cV} \right)_* \omega^{\kappa_\cV}(-D)=0, \text{ for all } q\geqslant 1,\]
where $\pi\times \mathbf{1}_{\cV}\colon \cX_0^\tor(\frakc,p^{\underline{n}},\underline{v})\times \cV \rightarrow \cX_0^\ast(\frakc,p^{\underline{n}},\underline{v})\times \cV $ is the natural projection (see also \cite[Thm.~8.2.1.3]{Lan}). 
It follows that the augmented  \v{C}ech complex, attached to 
an admissible affinoid covering of $\cY^{\tor}$ trivializing the 
invertible sheaf $\omega_m^{\kappa_\cV}(-D)$, is exact. 

Each term of  the complex being an orthonormalizable 
Banach $\cO(\cV)$-module, it is a consequence of the Open Mapping Theorem,  as explained in  \cite{Buz07}, that $\rH^0(\cY^{\tor}\times \cV, \omega_m^{\kappa_\cV}(-D))$ is  projective. 
\end{proof}

By the lemma $\RG(\cY^{\tor}_{\overline{J}_1,\ldots,\overline{J}_r,\ell,m}\times \cV, \omega_m^{\kappa_\cV}(-D))$ are concentrated in degree $0$, so the \v{C}ech complex
\[
\cC_{\overline{I}}^{\bullet}(\frakc,p^{\underline{n}}, \kappa_\cV, m)\colon 
\bigoplus_{J\supseteq I} \rH^0(\cY^{\tor}_{\overline{J},\ell,m}\times \cV, \omega_m^{\kappa_\cV}(-D)) \rightarrow
\bigoplus_{J_1,J_2\supseteq I}\rH^0(\cY^{\tor}_{\overline{J}_1,\overline{J}_2,\ell,m} \times \cV, \omega_m^{\kappa_\cV}(-D)) \rightarrow \cdots
\]
 is Hecke-equivariantly quasi-isomorphic to $\RG(\pi_{\HT, K_0(p^{\underline{n}})}^{\tor,-1}(]\overline{S}_{\overline{I}}[_m K_0(p^{\underline{n}}))\times \cV, \omega^{\kappa_\cV}(-D))$.

Recall that $\RG^{\overline{I}}(\frakc,p^{\underline{n}}, \kappa_\cV, \underline{v}) \defeq  \RG(\cX_0^\tor(\frakc,p^{\underline{n}},\underline{v})\times\cV,\omega^{\kappa_\cV}(-D))$. 
As  $\left\{\cX_0^\tor(\frakc,p^{\underline{n}},\underline{v})\right\}_{\underline{v}}$, where $\underline{v}\in(\QQ\cap[0,1])^{\Sigma_p}$ is such that $v_\frakq =1$ for all $\frakq\in \overline{I}$, and 
$\left\{\pi_{\HT, K_0(p^{\underline{n}})}^{\tor,-1}(]\overline{S}_{\overline{I}}[_m K_0(p^{\underline{n}}))\right\}_m$ 
   are cofinal,  we conclude that $\RG^{\overline{I}}(\frakc,p^{\underline{n}}, \kappa_\cV)
\defeq \varinjlim_v
\RG^{\overline{I}}(\frakc,p^{\underline{n}}, \kappa_\cV, \underline{v})$
is quasi-isomorphic to $\cC_{\overline{I}}^{\bullet}(\frakc,p^{\underline{n}}, \kappa_\cV)\defeq \varinjlim_m\cC_{\overline{I}}^{\bullet}(\frakc,p^{\underline{n}}, \kappa_\cV, m)$.
The quasi-isomorphism is Hecke-equivariant in the sense that the following diagram commutes for all Hecke operators $T_\frakl$ (including $U_\frakp$):
\begin{center}
\begin{tikzcd}
\RG^{\overline{I}}(\frakl\frakc,p^{\underline{n}}, \kappa_\cV) \ar[r, "\sim"] \ar[d,"T_\frakl"'] & \cC_{\overline{I}}^{\bullet}(\frakl\frakc,p^{\underline{n}}, \kappa_\cV) \ar[d,"T_\frakl"] \\
\RG^{\overline{I}}(\frakc,p^{\underline{n}}, \kappa_\cV) \ar[r, "\sim"] & \cC_{\overline{I}}^{\bullet}(\frakc,p^{\underline{n}}, \kappa_\cV).
\end{tikzcd}
\end{center}

We now turn our attention from $G^\ast$ to $G$.
Let  
\[\cC_{\overline{I},G}^{\bullet}(\frakc,p^{\underline{n}}, (\nu_\cU,w_\cU),m):=\left(\cC_{\overline{I}}^{\bullet}(\frakc,p^{\underline{n}}, \kappa_\cU,m)(-\nu_\cU)\right)^{\Delta_\frakn}\] be the complex of 
$\Delta_\frakn$-invariant of the $\nu_\cU$-twist of 
$\cC_{\overline{I}}^{\bullet}(\frakc,p^{\underline{n}}, \kappa_\cU,m)$. 
For $x\in F_+^{\times,(p)}$, the isomorphism $L_{\frakc,x}\colon X_{\frakc} \rightarrow X_{x\frakc}$ induces
\[
L_{\frakc,x}^\ast \colon \cC_{\overline{I},G}^{\bullet}(x\frakc,p^{\underline{n}}, (\nu_\cU,w_\cU),m) \rightarrow \cC_{\overline{I},G}^{\bullet}(\frakc,p^{\underline{n}},(\nu_\cU,w_\cU),m),
\]
which depends only on the fractional ideal $(x)$ but not on the generator $x$.
Define
\[
\cC_{\overline{I},G}^{\bullet}(p^{\underline{n}},  (\nu_\cU,w_\cU),m)\defeq
\bigoplus_{\frakc\in \Frac(F)^{(p)}}\cC_{\overline{I},G}^{\bullet}(\frakc,p^{\underline{n}}, (\nu_\cU,w_\cU),m)/\left(L_{\frakc,x}^\ast - \id\right)_{x\in F_+^{\times,(p)}}.
\] 
Again, 
$\RG_G^{\overline{I}}(p^{\underline{n}}, (\nu_\cU,w_\cU))$
is Hecke-equivariantly quasi-isomorphic to 
\[\cC_{\overline{I},G}^{\bullet}(p^{\underline{n}}, (\nu_\cU,w_\cU))\defeq \varinjlim_m\cC_{\overline{I},G}^{\bullet}(p^{\underline{n}}, (\nu_\cU,w_\cU), m).\]

\subsection{Slope decomposition}
Let $(A,A^+)$ be a complete Tate algebra over a non-Archimedean field.
Let $M$ be an $A$-module and let $T$ be an $A$-linear endomorphism of $M$.

\begin{definition} \label{def:hslopedecomp}
Let $h\in\QQ$. An \emph{$h$-slope decomposition of $M$ with respect to $T$} is a direct sum decomposition of $A$-modules $M=M^{\leqslant h}\oplus M^{>h}$
such that 
\begin{enumerate}
    \item $M=M^{\leqslant h}$ and $M^{>h}$ are stable under the $T$-action,
    \item $M^{\leqslant h}$ is a finite $A$-module,
    \item There is a unitary polynomial $Q\in A[X]$ with slope $\leqslant h$ such that $Q^\ast(T)$ is zero on $M^{\leqslant h}$,
    \item For any unitary polynomial $Q\in A[X]$ with slope $\leqslant h$, the restriction of $Q^\ast(T)$ to $M^{> h}$ is an invertible endomorphism.
\end{enumerate}
\end{definition}

Let $S=\Spa (A, A^+)$ and $\cM$ be a sheaf of $\cO_S$-modules.
Let $T\in \End_{\cO_S}(\cM)$.
\begin{definition}
We say that \emph{$\cM$ has a slope decomposition with respect to $T$} if for any $x\in S$ and $h\in \QQ$, there exists an affinoid neighborhood $U$ of $x$ in $S$, $h'\geqslant h$, and a $T$-stable decomposition of $\cO_U$-modules
\[
\left.\cM \right|_U = (\left.\cM \right|_U)^{\leqslant h'}\oplus (\left.\cM \right|_U)^{>h'},
\]
where $(\left.\cM \right|_U)^{\leqslant h'}$ is a coherent sheaf of $\cO_U$-modules and for any affinoid open $V = \Spa (B,B^+)\subseteq U$, 
\[
\cM(V) = (\left.\cM \right|_U)^{\leqslant h'}(V)\oplus (\left.\cM \right|_U)^{>h'}(V)
\]
is an $h'$-slope decomposition of the $B$-module $\cM(V)$.
\end{definition}

Suppose that $\cM$ has a slope decomposition with respect to $T$.
A section $s\in \cM(U)$ is said to \emph{have infinite slope} if for any affinoid open $V\subseteq U$ and $h\in\QQ$ such that $\cM(V)$ admits an $h$-slope decomposition, $\left. s\right|_V \in \cM(V)^{>h}$.
The infinite slope sections form a subsheaf $\cM^\infty\subseteq \cM$.
The \emph{finite slope part of $\cM$} is defined to be $\cM^\fs \defeq \cM/\cM^\infty$.
We can also construct the spectral variety $\cZ \hookrightarrow S\times\AA^1$ as follows.
For any affinoid open $U = \Spa (B,B^+)$ for which there is a slope decomposition $\cM(U) = \cM(U)^{\leqslant h} \oplus \cM(U)^{>h}$, define a map $B[X] \rightarrow \End_B(\cM(U)^{\leqslant h})$ by sending $X$ to $T^{-1}$. Let $I_h$ be the kernel of the map. 
Note that $I_h$ is generated by $Q$ guaranteed by (iii) in Definition~\ref{def:hslopedecomp}.
Let $\cZ_{U,h} \hookrightarrow S\times \AA^1$ be $\Spa (B[X]/I_h, (B[X]/I_h)^+)$, where $(B[X]/I_h)^+$ is the integral closure of $B^+$ in $B[X]/I_h$.
If $U'\subseteq U$ is an open affinoid, the flatness of $\cO_S(U) \rightarrow \cO_S(U')$ implies that $\cZ_{U,h}\times_{U} U' = \cZ_{U',h}$. If $h'\leqslant h$ and $\cM(U)$ also has an $h'$-slope decomposition, then $\cZ_{U,h} \hookrightarrow \cZ_{U,h'}$ is a union of connected components because $\cM(U) = \cM(U)^{\leqslant h} \oplus (\cM(U)^{\leqslant h'} \cap \cM(U)^{>h}) \oplus \cM(U)^{>h'}$.
Then $\cZ \defeq \coprod_{U,h}
 \cZ_{U,h}/\sim$, where $\sim$ is the equivalence relation identifying the domain with the image for the injections $\cZ_{U',h}\hookrightarrow \cZ_{U,h}$ and $\cZ_{U,h} \hookrightarrow \cZ_{U,h'}$. 
The projection $\cZ \rightarrow S$ is locally quasi-finite and partially proper.
There is a coherent sheaf $\cM_\cZ^{\fs}$, contained in the pullback of $\cM$ to $\cZ$, defined by $\cM_\cZ^{\fs}(\cZ_{U,h}) = \cM(U)^{\leqslant h}$.
The pushforward of $\cM_\cZ^{\fs}$ to $S$ is $\cM^{\fs}$, the finite slope part of $\cM$.

\begin{proposition}[{\cite[Prop.~6.1.11]{BP}}]
\label{prop:slope_decomp}
Let $M^\bullet$ be a bounded complex of projective Banach $A$-modules, and $\cM^\bullet$ the associated complex of Banach sheaves on $S$.
Let $T$ be a compact operator on $M^\bullet$ considered as an object in the derived category. Let $\rH^i(\cM^\bullet)$ be the $i$-th cohomology sheaf.
Then the sheaves $\rH^i(\cM^\bullet)$ have slope decomposition for $T$. Moreover, there is a complex $\cM^{\bullet,\fs}$ and a morphism $\cM^\bullet\rightarrow \cM^{\bullet, \fs}$ (unique up to a non unique quasi-isomorphism) such that $\rH^i(\cM^{\bullet,\fs}) = \rH^i(\cM^\bullet)^{\fs}$.
\end{proposition}

 Recall the \v{C}ech complex $\cC_{\overline{I},G}^{\bullet}(p^{\underline{n}}, (\nu_\cU,w_\cU))$ which by Lemma~\ref{lem:projective} is a bounded complex of projective Banach modules.
As  $\left\{\pi_{\HT, K_0(p^{\underline{n}})}^{\tor,-1}(]\overline{S}_{\overline{I}}[_m K_0(p^{\underline{n}}))\right\}_m$ and $\left\{\cX_0^\tor(\frakc,p^{\underline{n}},\underline{v})\right\}_{\underline{v}}$, where $\underline{v}\in(\QQ\cap[0,1])^{\Sigma_p}$ is such that $v_\frakq =1$ for all $\frakq\in \overline{I}$, 
   are cofinal, exactly as in the proof of  Lemma~\ref{lem:Upcompact} there exists $m'>m$ such that the $U_I$-action on the  $\cC_{\overline{I},G}^{\bullet}(p^{\underline{n}}, (\nu_\cU,w_\cU))$ seen as object in the derived category 
factors through the restriction maps 
\[\pi^{\tor,-1}_{\HT, K_0(p^{\underline{n}})}(]\overline{C}_{\overline{J}}[_{\overline{I},\ell,m}K_0(p^{\underline{n}}))
\to \pi^{\tor,-1}_{\HT, K_0(p^{\underline{n}})}(]\overline{C}_{\overline{J}}[_{\overline{I},\ell+1,m'}K_0(p^{\underline{n}}))
\] 
which are compact. Hence by  Proposition~\ref{prop:slope_decomp}, $\rH^i(\RG_G^{\overline{I}}(p^{\underline{n}}, (\nu_\cU,w_\cU))) \isom \rH^i(\cC_{\overline{I},G}^{\bullet}(p^{\underline{n}}, (\nu_\cU,w_\cU)))$ admits a  slope decomposition with respect to $U_I$. 

Despite the fact that Lemma~\ref{lem:Upcompact} ensures that $U_I$ is a compact operator on $\RG_G^{\overline{I}}(p^{\underline{n}},(\nu_\cU,w_\cU))$, we do need the Boxer--Pilloni strengthening of the Ash--Stevens slope decomposition theorem, because the terms of that complex are not {\it a priori} projective.

\subsection{The partial eigenvariety}
Let $A$ be an affinoid integral domain.
We say $A$ is \emph{relatively factorial} if for any $f=\sum_{n=0}^\infty a_n X^n\in A\langle X\rangle$ with $a_0=1$, $(f)$ factors uniquely as a product of principal prime ideals $(f_i)$ where each $f_i$ may be chosen with constant term $1$.
A rigid analytic space $\sW$ is relatively factorial if it has an admissible covering by relatively factorial affinoids.

A \emph{Fredholm series} is a global section $f\in \cO(\sW\times \AA^1)$ such that under the map 
\[ \cO(\sW\times \AA^1)\xrightarrow{j^\ast} \cO(\sW)\]
induced by $j\colon \sW\times \{0\} \rightarrow \sW\times \AA^1$ we have $j^\ast f=1$.
A \emph{Fredholm hypersurface} is a closed immersion $\cZ\subseteq \sW\times \AA^1$ such that the ideal sheaf of $\cZ$ is generated by a Fredholm series $f$, in which case we write $\cZ = \cZ(f)$.

\begin{definition}[{\cite[Def.~4.2.1]{Hansen}}]
An \emph{eigenvariety datum} is a tuple $\cD=(\sW, \cZ, \cM, \cT, \psi)$ where
\begin{itemize}
 \item $\sW$ is a separated, reduced, equidimensional, relatively factorial rigid analytic space,
 \item $\cZ\subseteq \sW\times \AA^1$ is a Fredholm hypersurface,
 \item $\cM$ is a coherent analytic sheaf on $\cZ$,
 \item $\cT$ is a commutative $\QQ_p$-algebra, and \
 \item $\psi\colon \cT\rightarrow \End_{\cO_\cZ}(\cM)$ is a $\QQ_p$-algebra homomorphism.
\end{itemize}
\end{definition}

\begin{theorem}[{\cite[Thm.~4.2.2]{Hansen}}]
\label{thm:eigenvariety_machine}
Given an eigenvariety datum $\cD$, there exists a separated rigid analytic space $\cE = \cE(\cD)$ together with a finite morphism $\pi\colon \cE \rightarrow \cZ$, a morphism $w\colon \cE \rightarrow \sW$, an algebra homomorphism $\cT \rightarrow \cO(\cE)$
 and a coherent sheaf $\cM^\dagger$ on $\cE$ together with a canonical isomorphism $\cM \isom \pi_\ast \cM^\dagger$ compatible with the $\cT$-actions on $\cM$ and $\cM^\dagger$.
 The points of $\cE$ lying over $z\in\cZ$ are in bijection with the generalized eigenspaces for the action of $\cT$ on $\cM(z)$.
\end{theorem}

We now construct the eigenvarieties from the cuspidal partially classical coherent cohomologies of Hilbert modular varieties.
Let $I\subseteq \Sigma_p$.
Let $\kappa_{\overline{I}}$ be a locally algebraic character on $\prod_{\frakq\in \overline{I}} \TT_{\frakq}(\ZZ_p)$.
To get an eigenvariety datum, we make the following choices:
\begin{itemize}
    \item $\sW = \cW_I^{G, \kappa_{\overline{I}}}$.
    \item $\cZ$ is the spectral variety associated to $\bigoplus_i \rH^i(\cC_{\overline{I},G}^{\bullet}(p^{\underline{n}}, (\nu_\cU,w_\cU)))$, which by construction is a Fredholm hypersurface.
    \item  $\cM$ is the finite slope part associated to $\bigoplus_i \rH^i(\cC_{\overline{I},G}^{\bullet}(p^{\underline{n}}, (\nu_\cU,w_\cU)))$.
    \item $\cT$ is the commutative $\cO_\cZ$-algebra generated by the operators $T_\frakl$ for $\frakl\nmid \frakn p$ and $U_\frakp$ for $\frakp\in I$.
    \item $\psi\colon\cT \rightarrow \End_{\cO_\cZ}(\cM)$ is the natural Hecke action.
\end{itemize}

\begin{theorem} \label{thm:main}
There is a separated rigid space $\cE_{I}^{\kappa_{\overline{I}}}$, called the $\overline{I}$-classical eigenvariety, equipped with the weight map $w\colon \cE_{I}^{\kappa_{\overline{I}}} \rightarrow \cW_{I}^{G,\kappa_{\overline{I}}}$ which is locally finite, and an algebra homomorphism $\psi\colon \cT \rightarrow \cO(\cE_{I}^{\kappa_{\overline{I}}} )$.
The points of $\cE_{I}^{\kappa_{\overline{I}}}$ are in bijection with the finite-slope generalized Hecke eigenspaces of $\RG_G^{\overline{I}}(p^{\underline{n}},(\nu_\cU,w_\cU))$.
\end{theorem}

For $I=\Sigma_p$, the partial eigenvariety $\cE_{\Sigma_p}$ is the cuspidal eigenvariety in \cite{AIP16}.

\begin{remark} \label{rmk:nonsurj}
After the construction of partial eigenvarieties, it is desirable to have functorialities between them for different $I\subseteq \Sigma_p$, such as
\[
\cE_I^{\kappa_{\overline{I}},I-\mathrm{fs}} \hookrightarrow 
 (\iota^G_{\kappa_{\overline{I}}})^\ast\cE_{\Sigma_p}
\]
where $\iota^G_{\kappa_{\overline{I}}}$ is the inclusion of weight spaces $\cW_{I}^{G,\kappa_{\overline{I}}} \hookrightarrow \cW^G$.
Even though we have natural maps between complexes of partially classical Hilbert modular forms by equation~\eqref{eq:functoriality}, it does not automatically provide functoriality between the partial eigenvarieties.
This is because the eigenvarieties are constructed by maximal dimension families, and it is not clear whether the specialization map of partially classical Hilbert modular families to lower dimensional families is surjective. Indeed, the corresponding result for overconvergent Hilbert modular forms makes crucial use of the fact that the ordinary locus in the minimal compactification $\cX^\ast$ is affine. In the finite slope case,  one possible approach would be to generalize the partial classicality theorem \cite{Hsu_partial} from classical weights to {\it partially} classical weights, which would  allow a comparison of the (finite slope part of the) partial eigenvariety to the full eigenvariety whose geometry is better understood. 
\end{remark}

\bigskip 
\subsection*{Acknowledgements} The research leading to this article was partially supported  by the Agence Nationale de Recherche grant ANR-18-CE40-0029. We would also like to thank the American Institute of Mathematics where part of this research was conducted. We are indebted to Adel Betina for numerous discussions regarding the geometric construction of eigenvarieties and for pointing us to a gap in an earlier version of this article related to the non-affineness of the partial ordinary locus, 
and to David Loeffler who generously shared his ideas regarding
affine coverings coming via the Hodge--Tate map which eventually filled the gap. We also heartily thank the anonymous referee for a detailed report.

\bibliographystyle{siam}

\end{document}